\documentclass[aap,preprint]{imsart}

\RequirePackage[OT1]{fontenc}
\RequirePackage{amsthm,amsmath}
\RequirePackage[numbers]{natbib}
\RequirePackage[colorlinks,citecolor=blue,urlcolor=blue]{hyperref}

\arxiv{arXiv:1708.04961}

\usepackage[utf8]{inputenc}
\usepackage[english]{babel}

\usepackage[top=3.cm, bottom=4.0cm, left=3.8cm, right=3.8cm]{geometry}

\startlocaldefs
\numberwithin{equation}{section}
\theoremstyle{plain}

\theoremstyle{plain}
\newtheorem{theorem}{Theorem}[section]
\newtheorem{lemma}[theorem]{Lemma}
\newtheorem{proposition}[theorem]{Proposition}
\newtheorem{corollary}[theorem]{Corollary}

\newtheorem{definition}[theorem]{Definition}

\newtheorem{remark}[theorem]{Remark}
\newtheorem{example}[theorem]{Example}
\newtheorem{assumption}[theorem]{Assumption}

\newcommand{\bE}{\mathbb{E}}

\newcommand{\bN}{\mathbb{N}}
\newcommand{\bP}{\mathbb{P}}

\newcommand{\bR}{\mathbb{R}}

\newcommand{\bZ}{\mathbb{Z}}

\newcommand{\cE}{\mathcal{E}}
\newcommand{\cF}{\mathcal{F}}

\newcommand{\cK}{\mathcal{K}}
\newcommand{\cL}{\mathcal{L}}

\newcommand{\cP}{\mathcal{P}}

\newcommand{\cS}{\mathcal{S}}

\newcommand{\cW}{\mathcal{W}}

\newcommand{\1}{\mathbbm{1}}
\endlocaldefs

\usepackage{indentfirst}
\usepackage{xcolor}

\usepackage{amsfonts}	
\usepackage{amssymb	
           ,bbm
           }
\usepackage{enumerate}

\begin{document}
\selectlanguage{english}

\begin{frontmatter}
\title{Freidlin-Wentzell LDP in path space for McKean-Vlasov equations and the Functional Iterated Logarithm Law
}
\runtitle{Freidling-Wentzell LDPs in path space for MV-SDEs}

\begin{aug}
\author{\fnms{Gon\c calo} \corref{} \snm{dos Reis}\thanksref{t1,m1}
\ead[label=e1]{G.dosReis@ed.ac.uk}
},
\author{\fnms{William} \snm{Salkeld}\thanksref{m1}
\ead[label=e2]{w.j.salkeld@sms.ed.ac.uk}}
\and
\author{\fnms{Julian} \snm{Tugaut}\thanksref{m2}
\ead[label=e3]{tugaut@math.cnrs.fr}
\ead[label=u3,url]{http://tugaut.perso.math.cnrs.fr}}

\thankstext{t1}{G. dos Reis acknowledges support from the \emph{Funda{\c c}$\tilde{\text{a}}$o para a Ci$\hat{e}$ncia e a Tecnologia} (Portuguese Foundation for Science and Technology) through the project UID/MAT/00297/2013 (Centro de Matem\'atica e Aplica\c c$\tilde{\text{o}}$es CMA/FCT/UNL).}
\runauthor{G. dos Reis et al.}

\affiliation{University of Edinburgh\thanksmark{m1} and Universit\'e Jean Monnet\thanksmark{m2}}

\address{University of Edinburgh\\ 
School of Mathematics \\
Peter Guthrie Tait Road \\
Edinburgh, EH9 3FD, UK\\  
\printead{e1}\\ 
\phantom{E-mail:\ }\printead*{e2}\\
}

\address{Universit\'e Jean Monnet\\
Institut Camille Jordan\\
23 Rue du Docteur Paul Michelon\\
42023 Saint-\'Etienne, France\\
\printead{e3}\\
\printead{u3}}
\end{aug}

\begin{abstract}{} We show two Freidlin-Wentzell type Large Deviations Principles (LDP) in path space topologies (uniform and H\"older) for the solution process of McKean-Vlasov Stochastic Differential Equations (MV-SDEs) using techniques which directly address the presence of the law in the coefficients and altogether avoiding decoupling arguments or limits of particle systems. We provide existence and uniqueness results along with several properties for a class of MV-SDEs having random coefficients and drifts of super-linear growth. 

As an application of our results, we establish a Functional Strassen type result (Law of Iterated Logarithm) for the solution process of a MV-SDE. 
\end{abstract} 



\end{frontmatter}

%
%
%

\section{Introduction}

In this article, we study a general class of McKean-Vlasov Stochastic Differential Equations (MV-SDEs) having drifts of polynomial growth and examine Freidlin-Wentzell type Large Deviations Principle small noise asymptotics in related path spaces, namely with the supremum- and H\"older-topologies.

MV-SDEs are more involved than classical SDEs as their coefficients depend on the law of the solution. They are sometimes referred to as mean-field SDEs and were first studied in \cite{McKean1966}. In a nutshell, these equations describe a limiting behaviour of individual particles having diffusive dynamics and which interact with each other in a ``mean-field'' sense. The analysis of stochastic particle systems and MV-SDEs interpreted as their limiting equations are of independent interest and appear widely in applications. Examples include molecular dynamics, fluid dynamics \cite{pope2000turbulent}; behaviour of large-scale interacting agents in economics or social networks or interacting neurons in biology see \cite{CarmonaDelarue2017book1,CarmonaDelarue2017book2} and references therein. Recently, there has been a vigorous growth in the literature on MV-SDEs addressing existence and uniqueness \cite{mishura2016existence}, smoothness of associated PDEs \cite{buckdahn2017mean,crisan2017smoothing}, numerical methods and many other aspects.
\medskip

We begin by reminding the reader what a Large Deviations Principle (LDP) is. The main goal of the large deviations is to calculate the probability of a rare event. In the case of stochastic processes, the idea is to find a deterministic path around which the diffusion is concentrated with high probability. As a consequence, the stochastic motion can be interpreted as a small perturbation of the deterministic path.

As a simple example, we present the idea of the large deviations principle for a classical diffusion with a constant coefficient diffusion:
\begin{equation}
\label{belier}
X_\varepsilon(t)=X(0)+{\sqrt{\varepsilon}} W(t)-\int_0^t\nabla V\left(X_\varepsilon(s)\right)ds\,,
\end{equation}
where $X(0)$ is deterministic, $W$ is a Brownian motion and $V$ is the so-called confining potential. We also introduce the deterministic path
\begin{equation}
\label{taureau}
\varphi(t)=X(0)-\int_0^t\nabla V\left(\varphi(s)\right)ds\,.
\end{equation}
Set $\nu_\varepsilon$ be the law of the diffusion $\left(X_\varepsilon(t)\right)_{t\in[0,1]}$. Then, we say that $\left(\nu_\varepsilon\right)_{\varepsilon>0}$ satisfies a large deviations principle on $C([0,1])$ equipped with the norm $\|\cdot\|_\infty$ with the rate function $I$ if and only if for any Borel set $\Gamma$, we have
\begin{equation*}
-\inf_{\phi\in\mathring{\Gamma}}I(\phi)
\leq
\liminf_{\varepsilon\to0} {\varepsilon}\log\big(\nu_\varepsilon(\Gamma)\big)
\leq
\limsup_{\varepsilon\to0} {\varepsilon}\log\big(\nu_\varepsilon(\Gamma)\big)
\leq
-\inf_{\phi\in\overline{\Gamma}}I(\phi)\,.
\end{equation*}
We will say that $I$ is a good rate function if the level set $\left\{x:I(x)\leq\alpha\right\}$ is compact for any $\alpha$. 

In the case of  \eqref{belier}, it is well known (see \cite{FreidlinWentzell2012}) that $\left(\nu_\varepsilon\right)_{\varepsilon>0}$ satisfies a LDP on $\big({C}([0,1]),\|\cdot\|_\infty\big)$ 
with the rate function $I$ defined as
\begin{equation*}
I(\phi):=\frac{1}{2}\int_0^1\big\|\,\dot\phi(t)+\nabla V\left(\phi(t)\right)\big\|^2dt\,,
\end{equation*}
if $\phi$ is absolutely continuous and such that $\phi(0)=X(0)$; we set $I(\phi):=+\infty$ otherwise. It follows that a Borel set $\Gamma$ of $C([0,1])$ which contains the deterministic path $\left(\varphi(t)\right)_{t\in[0,1]}$ in its interior is such that $\inf_{\phi\in\Gamma^0}I(\phi)=0$.
\medskip

We work with McKean-Vlasov SDEs satisfying $X_{\varepsilon}(0) = X_0$ and 
\begin{equation}
\label{eq:BasicMVSDE}
dX_\varepsilon(t)
=
b_{\varepsilon}\Big(t,X_\varepsilon(t),\cL^{X_{\varepsilon}}_t\Big)dt
+{\sqrt{\varepsilon}} \sigma_{\varepsilon}\Big(t,X_\varepsilon(t),\cL^{X_{\varepsilon}}_t \Big)dW(t),
\end{equation}
where $\cL^{X_{\varepsilon}}_t$ stands for $\textrm{Law}(X_\varepsilon(t))$. Since the law of the process is present in the coefficients, this equation is nonlinear - in the sense of McKean. Exact assumptions on $\sigma_{\varepsilon}$ and on $b_{\varepsilon}$ will be given subsequently. Let us discuss a particular case of this McKean-Vlasov diffusion (in dimension one):
\begin{equation}
\label{eq:intro:init}
X_{\varepsilon}(t)
=
X_{\varepsilon}(0)
+{\sqrt{\varepsilon}} W(t)-\int_0^tV'\big(X_\varepsilon(s)\big)ds
-\int_0^tF'\ast u_s^\varepsilon\left(X_\varepsilon(s)\right)ds,
\end{equation}
where $u_s^\varepsilon = \cL^{X_{\varepsilon}}_s$, $\sigma_{\varepsilon}(t,x,\mu):=1$, $b_{\varepsilon}(t,x,\mu):=-V'(x)-F'\ast\mu(x)$ and ``$\ast$'' is the usual convolution operator.

The motion of the process is generated by three concurrent forces. The first one is the derivative of a potential $V$ - the confining potential. The second influence is a Brownian motion $\left(W(t)\right)_{t\in\mathbb{R}_+}$. It allows the particle to move against the potential $V$. The third term - the so-called self-stabilizing term - represents the attraction between all the others trajectories. Indeed, we remark:
\begin{equation*}
F'\ast u_s^\varepsilon\big(X_\varepsilon(s)(\omega_0)\big)
=\int_{\omega\in\Omega}F'\Big(X_\varepsilon(s)(\omega_0)-X_\varepsilon(s)(\omega)\Big)d\mathbb{P}\left(\omega\right),
\end{equation*}
where $\left(\Omega,\mathcal{F},\mathbb{P}\right)$ is the underlying measurable space, see \cite{McKean,McKean1966}.

The particle $X_{\varepsilon}$ which verifies \eqref{eq:intro:init} can be seen as one particle in a continuous mean-field system of an infinite number of particles. The mean-field system that we will consider is a random $N$-dimensional dynamical system for $i\in\{1,\cdots,N\}$ where $X^{i,N}_\varepsilon(0) = X_0$ and 
\begin{equation*}
dX_{\varepsilon}^{i,N}(t)
=
{\sqrt{\varepsilon}} dB^{i,N}(t)-V'\left(X_{\varepsilon}^{i,N}(t)\right)dt-\frac{1}{N}\sum_{j=1}^NF'\left(X_{\varepsilon}^{i,N}(t)-X_{\varepsilon}^{j,N}(t)\right)dt, 
\end{equation*}
where the $N$ Brownian motions $\big(B^i(t)\big)_{t\in\mathbb{R}_+}$ are independent. Mean-field systems are the subject of a rich literature. The link between the self-stabilizing process and the mean-field system when $N$ goes to $+\infty$ is called Propagation of Chaos.  We say that there is \emph{propagation of chaos} for the system of interacting particles when the law of $k$ fixed particles $X^{i,N}$ tends to the distribution of $k$ independent particles $X$ solving \eqref{eq:intro:init} with same law when the size of the system $N$ goes to infinity, see \cite{Sznitman} under Lipschitz properties; \cite{M1996} under Lipschitz assumptions but allowing for jumps; \cite{BRTV} if $V$ is a constant; \cite{CGM} for a uniform result in time in the non-uniformly convex case. For applications, see \cite{CDPS2010} about social interactions or \cite{CX2010} about the stochastic partial differential equations. 

Another side to Propagation of Chaos are Large deviations results which quantify the rate of convergence of the empirical measure in exponential scales. Many LDP results for McKean-Vlasov SDEs exist exploring Sanov type large deviations for the $N$-particle empirical measures from the McKean-Vlasov limit. This is a huge field and a small selection of relevant references is given by \cite{DawsonGaertner1987-DG1987,dai1996mckean,budhiraja2012large,deuschel2017enhanced} (see references therein). As argued in \cite{budhiraja2012large}, these results are a kind of Freidlin-Wentzell small noise asymptotics, but they are ``small noise'' at a different level (that of measure-valued processes or path-distribution-valued random variables) compared to the usual (process level) Freidlin-Wentzell results being discussed in this work.

\subsection*{Our contributions}

We prove our results by dealing with the presence of the laws in the coefficients directly and avoiding arguments on empirical measures or approximation/convergence of measures. Moreover, our LDP result in the H\"older-topology sharpen existing ones in the classical SDE framework.

\smallskip

\textbf{Existence and uniqueness problem.} The existence problem for \eqref{eq:intro:init} has been investigated by two different methods. The first one consists in the application of a fixed point theorem, see \cite{McKean,BRTV,CGM,HerrmannImkellerPeithmann2008} in the non-convex case. The other consists in a propagation of chaos, see for example \cite{M1996}. Moreover, it has been proved in \cite[Theorem 2.13]{HerrmannImkellerPeithmann2008} that there is a unique strong solution. Further results on existence and uniqueness, but away from our setting, can be found in \cite{Carmona2016Lectures,CarmonaDelarue2017book1,CarmonaDelarue2017book2,mishura2016existence}. We highlight \cite{Scheutzow1987} for a discussion on counterexamples on uniqueness of solutions.


We work with MV-SDE with dynamics \eqref{eq:BasicMVSDE} and the work closest to ours is that of \cite{Gaertner1988}. There the author provides existence and uniqueness results for \eqref{eq:BasicMVSDE} under similar super-linear growth conditions but his methodology focuses on certain PDE arguments which force the coefficients to be deterministic, time-independent and impose a uniform ellipticity constraint on $\sigma$. Our methods are fully probabilistic in nature and lift these restrictions. We assume a random drift $b$ with of spatial superlinear growth satisfying a (non-coercive) monotonicity condition and a random possibly degenerate Lipschitz diffusion coefficient (see Assumption \ref{ass:MKSDE-MainExistTheo} below). 
\medskip

\textbf{Large Deviations.} The LDP results we present are with respect to the vanishing noise (when $\varepsilon \searrow 0$ in \eqref{eq:BasicMVSDE}) as in Freidlin-Wentzell theory. For instance, \cite{HerrmannImkellerPeithmann2008} investigates the large deviations principle for the McKean-Vlasov diffusion~\eqref{eq:intro:init} in general dimensions, assuming superlinear growth of the drift but imposing coercivity in their monotonicity condition and a constant diffusion term. In particular, they show that the family of laws $(\nu_\varepsilon)_\varepsilon$ satisfies a large deviations principle on $C([0,1])$ equipped with the uniform norm with the good rate function
\begin{equation*}
I(\phi):=\frac{1}{2}\int_0^1\big\|\dot\phi(t)+\nabla V\left(\phi(t)\right)+F'\left(\phi(t)-\varphi(t)\right)\big\|_\infty^2dt\,,
\end{equation*}
when $\phi$ is absolutely continuous such that $\phi(0)=X(0)$ and $I(\phi):=+\infty$ otherwise ($\varphi$ as in \eqref{taureau}). 

We show a similar result, in the uniform norm, for the family associated to \eqref{eq:BasicMVSDE}. However, unlike \cite{HerrmannImkellerPeithmann2008}, we assume a Lipschitz $\sigma$ coefficient (not a constant one) and we do not impose any coercivity condition (strict negativity of the monotonicity constant). For this result we combine aspects of their work jointly with \cite{DemboZeitouni2010}. 

Concerning the H\"older topologies LDP, we find inspiration in \cite{arous1994grandes}. Studying standard SDEs, the authors find a way to transfer LDP results from a coarse topology to a finer one; in their case, from supremum norms to H\"older norms. Their method, explained later, relies on establishing the following inequality: $\forall R>0, \forall \rho>0$, $\exists \delta >0$ and for $\varepsilon$ small enough (see Theorem \ref{HolderProp} below for the precise statement) 
\begin{equation}
\label{eq:Section1Skull}
\bP \Big[ \|X_{\varepsilon}^x - \Phi^x(h) \|_\alpha \geq \rho,
          \|{\sqrt{\varepsilon}} W - h\|_\infty \leq \delta
		\Big] 
			    \lesssim \exp\Big( -\frac{R}{\varepsilon}\Big),
\end{equation}
for classical SDE's where $\Phi^x(h)$ is the so-called Skeleton map (an ODE) associated with $X^X_{\varepsilon}$. This can be thought of as establishing that the probability of $X$  having a high variation in the $\|\cdot\|_\alpha$-norm given that the input signal (from the Brownian motion) is small in $\|\cdot\|_\infty$-norm is exponentially small. For this, they assume boundedness and Lipschitz properties of the drift and diffusion coefficients of the SDE $X_{\varepsilon}$ dependent only on the spatial variables.  We provide results in the same vein but for the general class of MV-SDEs with drifts of polynomial growth (see Assumption \ref{def:He-man+Skeletor}). Their conditions are stronger than our conditions so our results extend existing results in classical SDE literature. To the best of our knowledge LDPs in path space in H\"older topologies or general Besov-Orlicz spaces for MV-SDEs remains unexplored.

Our results on LDPs are of general interest and can be applied to the Monte-Carlo simulation of MV-SDEs. They can be used in the spirit of \cite{GuasoniRobertson2008} as a way to find the optimal Importance Sampling measure, see \cite{dRST2018ImportanceSampling}. 

\medskip

\textbf{Functional Strassen Law}
The final contribution of our work is a Functional Strassen Law (a type of Law of Iterated Logarithm) for the solution of a MV-SDE. Strassen's Law for a Brownian motion $W$ was originally stated in \cite{strassen1964invariance} and says for any $t$
\begin{align*}
\frac{W(nt)}{\sqrt{n\log\log(n)}} \to 0 \textrm{ in Prob.~as $n\to \infty$ but a.s.~convergence does not hold}.
\end{align*}
We show that if one replaces the Brownian motion $W$ by the solution of a MV-SDE then the result still holds; a by-product is that this statement allows one to characterize Lyapunov functions for such equations. 

In broad strokes, the essence of our proof stems from arguments in \cite{baldi1986large} and is about showing that the set of rescaled paths is relatively compact in the path space topology which implies convergence in probability, but that the set of limit points of this set (connected to the Skeleton of MV-SDEs) is uncountable which implies the failure of almost sure convergence. In \cite{baldi1992large} the authors show Strassen's result in H\"older topologies for the Brownian motion. 

The work closest to ours is \cite{baldi1986large}. A similar result is shown for standard SDEs with time-independent uniformly Lipschitz and bounded coefficients. These conditions are much stronger than the conditions we impose (roughly our Assumption \ref{ass:MKSDE-MainExistTheo} but with $b,\sigma$ deterministic, time-independent and $\sigma$ bounded), and hence our results  extends the existing results of the classical SDE literature.

In methodological terms, we recast the notion of the re-scaling operators used in \cite{baldi1986large} to fit the MV-SDE setting (our Definition 5.2 \& 5.4) and most notably so that they act on the process and law in tandem and in the ``right way''. After this build-up, we prove our main result as described above.

\bigskip

\textbf{Our contribution in view of the ``decoupling argument''.} From a methodological point of view, many results of standard SDEs can be carried forward to the MV-SDE framework using the so-called ``decoupling argument''. The latter, is just that after establishing existence and (crucially) uniqueness for a MV-SDE, one can freeze the law (via an independent copy) and the dynamics that remains is that of a standard SDE with an added time dependency. As long as the new time dependency has the right properties one can apply most of the known results of standard SDE and transfer them to the MV-SDEs setting. Concerning the LDPs, this topic is discussed in Section \ref{Section:DecouplingArgument}; concerning the Functional Strassen results see Remark \ref{rem:NoDecoupling}.

\bigskip
This work is organized as follows.

 In Section \ref{sec:preliminaries} we introduce this work's notation and in Section~\ref{sec:MV-SDEsexist+uniq} we prove the existence/uniqueness results as well as deriving properties of the associated dynamics. The LDP results appear in Section \ref{sec:LPDresults}. In Section \ref{sec:StrassenSection}  we establish a Functional Iterated Logarithm law (Strassen's law) for the solution of the MV-SDE. Some auxiliary results, including extensions of some other known ones, are provided in the Appendix.

\section{Preliminaries}
\label{sec:preliminaries}

\subsection{Notation}

We denote by $\bN=\{1,2,\cdots\}$ the set of natural numbers; $\bZ$ and $\bR$ denote the set of integers and reals respectively; $\bR^+=[0,\infty)$. By $a \lesssim b$ we denote the relation $a \leq C\, b$ where $C>0$ is a generic constant independent of the relevant parameters and may take different values at each occurrence. By $\lfloor x \rfloor$ we denote the largest integer less than or equal to $x$. Let $A$ be a $d\times d'$ matrix, we denote the Transpose of $A$ by $A^T$. 

Let $f:\bR^d \to \bR$ be a differentiable function. Then we denote $\nabla f$ to be the gradient operator and $H[ f]$ to be the Hessian operator. $\partial_{x_i}$ is the 1st partial derivative with relation to the $i$-th position.

\subsubsection*{Probability}

Let $0<T<\infty$. Let $(\Omega, \cF, \bP)$ be a probability space carrying a $d'$-dimensional Brownian Motion on the interval $[0,T]$. The Filtration on this space satisfies the usual assumptions. We denote by $\bE$ and $\bE[\cdot|\cF_t]$ the usual expectation and conditional expectation operator respectively. For a random variable $X$ (RV in short) we denote its probability distribution (or Law) by $\cL^X=\bP \circ X^{-1}$; the law of a process $(Y(t))_{t\in[0,T]}$  at time $t$ is denoted by $\cL_t^Y=\bP \circ [Y(t)]^{-1}$.

Let $L^{p}(\cF_t,\bR^d)$, $t\in [0,T]$, is the space of $\bR^d$-valued $\cF_t$-measurable RVs $X$ with norm  $\|X\|_{L^p} = \bE[\, |X|^p]^{1/p} < \infty$; $L^\infty$ refers to the subset of essentially bounded RVs. $\cS^{p}([0,T])$ is the space of $\bR^d$-valued measurable $\cF$-adapted processes $Y$ satisfying $\|Y \|_{\cS^p} =
\bE[\sup_{t\in[0,T]}|Y(t)|^p]^{1/p}
<\infty$; $\cS^\infty$ refers to the subset of $\cS^{p}(\bR^d)$ of absolutely
uniformly bounded processes. 


\subsubsection*{Other spaces and norms}

We set $C([0,T])$ as the space of continuous functions $f:[0,T]\to \bR$ endowed with the uniform norm $\|\cdot\|_\infty$. 
For the space of continuous functions on the interval $[0,T]$ and $\alpha\in(0,1)$, we define the uniform and the H\"older norm of a function $\psi$
\begin{align*}
\|\psi\|_\infty = \sup_{t\in[0,T]} |\psi(t)| 
\qquad \textrm{and}\qquad
\|\psi\|_\alpha = \sup_{s,t\in [0,T]} \frac{|\psi(t)-\psi(s)|}{|t-s|^\alpha}.
\end{align*}
With $\|\cdot\|_\alpha$ we define the space of $\alpha$-H\"older continuous functions $f:[0,T] \to \bR$ by $C^\alpha([0,T],\bR)$; a ball centered on the map $\psi$ and with radius $r>0$ in this topology $C^\alpha([0,T],\bR)$ is denoted as $B_\alpha(\psi, r)$; we use $B_\infty(\psi, r)$ to denote the same ball on the topology of the $\|\cdot\|_\infty$-norm. 

We define, for $t\in [0,T]$, the restricted norm $\|\cdot\|_{\alpha, t}$ and $\|\cdot\|_{\infty, t}$ based on $\|\cdot\|_{\alpha}$ and $\|\cdot\|_{\infty}$ such that $\|\cdot\|_{\alpha, T}=\|\cdot\|_{\alpha}$, $\|\cdot\|_{\infty, T}=\|\cdot\|_{\infty}$ and is defined as
\begin{align*}
\|f\|_{\infty, t} =\sup_{0\leq s \leq t} |f(s)|
\qquad\textrm{and}\qquad
\|f\|_{\alpha, t} = \sup_{0\leq r<s\leq t} \frac{|f(s)-f(r)|}{|s-r|^\alpha}.
\end{align*}
This is similar to the H\"older/Supremum  norm and they are also a monotone increasing function with respect to $t$. It is also clear that  $\forall\, \psi\in C^\alpha([0,T])$ with $\psi(0)=0$ we have $\|\psi\|_{\infty,t}\leq \|\psi\|_\infty \leq \|\psi\|_{\alpha}$ and $\|\psi\|_{\infty,t} \leq \|\psi\|_{\alpha,t} \leq \|\psi\|_{\alpha}$ (see Lemma \ref{Lemma3}). 

Let $L^2([0,T])$ denote the space of square integrable functions $f:[0,T]\to \bR$ satisfying $\|f\|_2:= \big(\int_0^T |f(r)|^2dr\big)^{1/2}<\infty$. Let $H$ be the usual Cameron-Martin Hilbert space for Brownian motion; the space of all absolutely continuous paths on the interval $[0,T]$ which start at $0$ and have a derivative almost everywhere which is $L^2([0,T])$ integrable, 
$$
H:= \Big\{h:[0,T]\mapsto \bR,\ h(0)=0,\ h(\cdot)=\int_0^\cdot \dot{h}(s)ds ;\ \dot{h}\in L^2([0,T])\Big\}.
$$
It is easy to see that if $h\in H$ then $h(0)=0$,  $h\in C^{\frac12}([0,T])$ and $\|h\|_\infty \leq \|h\|_{\frac12} \leq \|\dot h\|_{2}$.


\subsection{The Wasserstein metric}

In this section we introduce the Wasserstein metric and some results related to it, for in-depth treatments we refer the reader to \cite{Villani2009} or \cite[Chapter 5]{CarmonaDelarue2017book1}.  Consider a measurable space $(E,\cE)$ and let $\cP(E)$ be the class of probability measures in this space. Let $k\in \bN$, let $\cP_k(E)$ be the space of probability distributions on $(E, \cE)$ with finite $k$-th moments. The Dirac delta measure concentrated at a point $x\in E$ is denoted by $\delta_x$. We define a metric on the space of distributions. 
\begin{definition}[Wasserstein metric]
Let $E$ be a complete, separable metric space with metric $d:E\times E \to \bR^+$ and $\sigma$-algebra $\cE$. Let $\mu, \nu \in\cP_2(E)$. We define the Wasserstein distance to be
$$
W^{(2)}(\mu, \nu) = \inf\Big\{ \Big( \int_{E^2} d(x, y)^2\pi(dx, dy)\Big)^{1/2}; \pi\in \cP(E\times E)\Big\},
$$
where $\mu(A) = \int_{E^2} \chi_A(x) \pi(dx,dy)$ and $\nu(B) = \int_{E^2} \chi_B(y) \pi(dx,dy)$. 
\end{definition}
The Wasserstein metric is a metric and it induces a topology on $\cP_2(E)$. This has been shown to be the topology of weak convergence of measure together with the convergence of all moments of order up to $2$. It is important to define the Wasserstein distance for a generic complete separable metric space because later on we will be interchanging between measures on $\bR^d$ and $C([0,T];\bR^d)$. In order to distinguish between these two types of objects, we denote $m\in \cP_2(C([0,T];\bR^d))$ and $m_t\in \cP_2(\bR^d)$ and we define for $A\subset \bR^d$
$$
m_t(A) = \int_{C([0,T];\bR^d)} \1_{\big\{x(\cdot)\in C([0,T];\bR^d);\, x(t)\in A\big\}}(x)m(dx).
$$

If one needed a metric on the entire space $\cP(E)$ rather than the subset $\cP_2(E)$, one could use the \emph{Modified Wasserstein Distance}
$$
W^{(0)} (\mu, \nu) = \inf\Big\{ \int_{E^2} \Big[1\wedge d(x, y)\Big] \pi(dx, dy); \pi\in \cP(E\times E)\Big\}.
$$
where $\pi$ has marginals $\mu$ and $\nu$ as before. This metric induces that of weak convergence on $\cP(E)$. 

\begin{definition}
Let $\cP_2(E)$ be the set of all probability distributions (measures) on the separable vector space $E$ with finite second orders. Endow this set with two operators called addition 
$+_{\cP_2}:\cP_2(E)\times \cP_2(E) \to \cP_2(E)$ and scalar multiplication $\times_{\cP_2}:\bR^d \times \cP_2(E)\to \cP_2(E)$ such that $\forall \mu, \nu\in \cP_2(E)$, $c\in \bR^d$ and $A\subset E$ we have
\begin{align*}
(\mu +_{\cP_2} \nu)[A]=\int_E \mu(y-A) \nu(dy) 
\qquad \textrm{and}\qquad
( c \times_{\cP_2} \mu)[A]= \mu\Big[\frac{A}{c}\Big].
\end{align*}
These operators satisfy the vector axioms and so they form a vector space. 
\end{definition}

These vector operators are more intuitive if one thinks of the set of probability distributions as the set of all random variables on $E$. These measure operators represent addition and scalar multiplication of independent random variables with respect to the vector operators within the space $E$. This vector space, like all vector spaces, has a $0$ element. This is the delta distribution centered at the $0$ element of $E$, $\delta_0$. The convolution of the delta distribution with any other measure is that measure and it remains constant under stretches and compressions of the domain centered around $0$. 


The next result is a simple computation which we evaluate for the benefit of the reader.
\begin{lemma}
\label{lemma:WassersteinAgainstDirac}
Take $\delta_0$. Then for any $\mu \in \cP_2(E)$ we have $W^{(2)}(\mu, \delta_0) = \big(\int_E y^2 \mu(dy)\big)^{1/2}$. 
\end{lemma}
\begin{proof}
Consider a random variable with law $\delta_0$. We have $X:\Omega \to E$ with $\bP[X\in A]=\delta_0(A)$ for any $A\subset E$. The $\sigma$-algebra generated by $X$ is just $\{\Omega, \varnothing\}$. 
Let $\mu\in\cP_2[E]$ be the law of a random variable $Y:\Omega \to E$ which generates a $\sigma$-algebra that $X$ will be measurable with respect to. For any $B\in \sigma(Y)$, we have that 
$\bP[\Omega \cap B] = \bP[B] = 1 \bP[B] = \bP[\Omega]\bP[B]$ and $\bP[B\cap \varnothing] = \bP[\varnothing] = 0 = \bP[B]\bP[\varnothing]$. Hence $\sigma(X)$ and $\sigma(Y)$ are independent. 
Therefore we have that the joint density function of $X$ and $Y$ is just $\mu(dy)\delta_0(dx)$ and the conclusion follows. 
\end{proof}



\section{McKean-Vlasov equations with locally Lipschiz coefficients}
\label{sec:MV-SDEsexist+uniq}

\subsection{Existence and Uniqueness of Solutions}


We start with a slight generalization of the existence and uniqueness result under Lipschitz conditions in \cite[Theorem 1.7]{Carmona2016Lectures}. Let $W$ be a $d'$-dimensional Brownian motion and take the progressively measurable maps $b:[0,T] \times \Omega \times \bR^d \times\cP_2(\bR^d) \to \bR^d$ and $\sigma:[0,T] \times \Omega \times \bR^d \times \cP_2(\bR^d) \to \bR^{d\times d'}$.

We introduce, for $0\leq t\leq T<\infty$ the dynamics of a process $Y$ as
\begin{align}
\label{eq:MKSDE-MainExistTheo}
dY(t) = b\big(t,Y(t), \cL_t^Y\big)dt + \sigma\big(t,Y(t), \cL_t^Y\big)dW(t),
\end{align}
for $Y(0) \in L^p(\cF_0; \bR^d; \bP)$ and where $\cL_t^Y$ denotes the Law of $Y(t)$.
\begin{theorem}
Suppose that $b$ and $\sigma$ are integrable in the sense that 
\begin{align*}
\bE\Big[\Big( \int_0^T |b(t, \omega, 0, \delta_0)| dt\Big)^2 \Big] <\infty 
\quad \textrm{and}\quad
\bE\Big[\int_0^T |\sigma(t, \omega, 0, \delta_0)|^2 dt \Big] <\infty ,
\end{align*}
 and Lipschitz in the sense that $\exists L>0$ such that $\forall t \in[0,T]$, $\forall \omega \in \Omega$, $\forall x, x'\in \bR^d$ and $\forall \mu, \mu'\in \cP_2(\bR^d)$ we have that
\begin{align*}
|b(t, \omega, x, \mu)-b(t, \omega, x',\mu')| + |\sigma(t, \omega, x, \mu)&-\sigma(t, \omega, x', \mu')|\\&\leq L(|x-x'| + W^{(2)}(\mu, \mu') ).
\end{align*}
Suppose further that $X(0)\in L^2(\Omega, \cF_0, \bP; \bR^d)$ is a square integrable random variable which is independent of the Brownian motion. Then there exists a unique solution for $Y\in \cS^2([0,T]; \bR^d)$ to the MV-SDE \eqref{eq:MKSDE-MainExistTheo} and $\cL_0^Y\in \cP_2(\bR^d)$ where $\cL_t^Y$ is the probability distribution of the random variable $Y(t)$. 
\end{theorem}
\begin{proof}
For $b(\cdot,0,\delta_0)$ satisfying
$\bE[\int_0^T |b(t, \omega, 0, \delta_0)|^2 dt ] <\infty$ the result is known e.g.~\cite[Theorem 1.7]{Carmona2016Lectures}. A close inspection of that proof shows that this condition is not sharp. In particular, the result holds with the slightly weaker integrability condition found in the statement of the theorem we present here. The verification is straightforward and we do not carry it out.
\end{proof}

We extend the previous result to the locally Lipschitz case, see \cite{mishura2016existence} for other results. We work with general monotonicity assumptions without imposing coercivity restrictions. We also sharpen the integrability assumptions and leave it to the reader to verify that the proof in \cite{Carmona2016Lectures} can be sharpened. 
\begin{assumption}
\label{ass:MKSDE-MainExistTheo}
Let $p\geq 2$. The progressively measurable maps $b:[0,T] \times \Omega \times \bR^d \times\cP_2(\bR^d) \to \bR^d$ and $\sigma:[0,T] \times \Omega \times \bR^d \times \cP_2(\bR^d) \to \bR^{d\times d'}$ satisfy that $\exists L>0$ such that: 
\begin{enumerate}
\item $Y(0) \in L^p(\cF_0; \bR^d; \bP)$ be independent of the Brownian motion. 
\item Integrability: $b$ and $\sigma$ satisfy
$$
\bE\Big[\Big(\int_0^T |b(t, 0, \delta_0)| dt\Big)^p\Big], \bE\Big[ \int_0^T \Big| \sigma(t, 0, \delta_0) \Big|^2 dt\Big)^{\tfrac{p}{2}} \Big] <\infty, 
$$ 
\item $\sigma$ is Lipschitz: $\forall t \in[0,T]$, $\forall x, x'\in \bR^d$ and $\forall \mu, \mu' \in \mathcal{P}_2(\bR^d)$ we have
\begin{align*}
|\sigma(t, x, \mu) - \sigma(t, x', \mu')| \leq L\Big( |x-x'| +W^{(2)}(\mu, \mu')\Big),
\end{align*}
\item $b$ satisfies the monotone growth condition in $x$ and is Lipschitz in $\mu$: $\forall t \in[0,T]$, $\forall x,x'\in \bR^d$ and $\forall \mu, \mu' \in \cP_2(\bR^d)$ we have that
\begin{align*}
\langle x-x', b(t, x, \mu)-b(t, x', \mu)\rangle_{\bR^d} & \leq L|x-y|^2\\
\textrm{and}\quad
|b(t, x, \mu) - b(t, x, \mu')| & \leq L W^{(2)} (\mu, \mu'),
\end{align*}
\item $b$ is Locally Lipschitz with Polynomial Growth in $x$: $\exists q\in \bN$ such that $q>1$ and $\forall t \in[0,T]$, $\forall \mu \in \cP_2(\bR^d)$, $\forall x, x'\in\bR^d$ we have
$$
|b(t, x, \mu) - b(t, x', \mu) | \leq L (1+|x|^{q-1}+|x'|^{q-1}) |x-x'|.
$$
\end{enumerate}
\end{assumption}

\begin{theorem}
\label{theo:LocLipMcK-V.ExistUniq}
Let $p\geq 2$. Recall the dynamics of $Y$ given by \eqref{eq:MKSDE-MainExistTheo}, where the drift and diffusion coefficients $b,\sigma$ and initial RV $Y(0)$ satisfy Assumption \ref{ass:MKSDE-MainExistTheo} with $p\geq 2$. 
Then there exists a unique solution $Y\in \cS^{p}\cap \cS^{2}$ to \eqref{eq:MKSDE-MainExistTheo} and $\exists C>0$ such that
\begin{align*}
\bE\Big[\sup_{t\in[0,T]}|Y(t)|^{p}\Big] \leq C \Big(\bE[|Y(0)|^{p}] & + \bE\Big[ \Big( \int_0^T |b(t, 0, \delta_0)| dt\Big)^p\Big] \\
&+ \bE\Big[ \Big( \int_0^T |\sigma(s, 0, \delta_0)|^2ds \Big)^{\tfrac{p}{2}} \Big] \Big)e^{CT}.
\end{align*}
\end{theorem}

\begin{proof}
Consider the operator
$$
\Xi: \cP_2(C([0,T], \bR^d)) \to \cP_2(C([0,T], \bR^d)),
$$
where $\Xi(\mu) = \cL^{Y^\mu}$ denotes the law of the SDE's solution $Y^\mu$ with dynamics
$$
dY^\mu(t) = b(t, Y^\mu(t), \mu_t)dt + \sigma(t, Y^\mu(t), \mu_t)dW(t), \quad Y^\mu(0) = Y(0). 
$$
We start by showing that given some $\mu$, a solution to the above SDE exists. Let $\mu \in \cP_2(C([0,T], \bR^d))$. Define 
$$
\hat{b}^\mu(t, x) = b(t, x, \mu_t), \quad \hat{\sigma}^\mu(t, x) = \sigma(t, x, \mu_t).
$$
Then we have
\begin{align*}
& \bE\Big[ \Big( \int_0^T |\hat{b}^\mu(t, 0)| dt \Big)^{p} \Big]
\\
& \qquad 
 \leq \bE\Big[ \Big( \int_0^T |b(t, 0, \delta_0)| + L \cdot W^{(2)}(\mu_t, \delta_0)  dt \Big)^{p} \Big]
\\
& \qquad \leq 2^{p-1} \bE\Big[ \Big( \int_0^T |b(t, 0, \delta_0)| dt\Big)^p\Big]  + 2^{p-1}L^pT^p \cdot W^{(2)} ( \mu, \delta_0)^p < \infty,
\end{align*}
and similarly 
\begin{align*}
& \bE\Big[ \Big( \int_0^T |\hat{\sigma}^\mu(t, 0)|^2 dt \Big)^{\tfrac{p}{2}} \Big]
\\
&\qquad  \leq 2^{p-1} \bE\Big[ \Big( \int_0^T |\sigma(t, 0, \delta_0)|^2 dt\Big)^{\tfrac{p}{2}} \Big]  + 2^{p-1}L^pT^p \cdot W^{(2)} ( \mu, \delta_0)^p < \infty. 
\end{align*}
Also we have that $\hat{b}^\mu(t, x)$ is locally Lipschitz, satisfies a monotone growth condition and $\hat{\sigma}^\mu(t, x)$ has Lipschitz growth in its spacial variables. Therefore, by the methods in \cite[Theorem 3.6]{mao2008stochastic}, we have that a unique solution exists in $\cS^p([0,T])$. Since $p\geq2$, we can conclude that $\cL^{Y^\mu}\in \cP_2(C([0,T], \bR^d))$. 

Using It\^o's formula, we have that
\begin{align}
\nonumber
 W^{(2)} &( \Xi(\mu) , \Xi(\nu))^2 \leq \bE\Big[ \sup_{t\in[0,T]} |Y^{\mu}(t) - Y^{\nu}(t)|^2 \Big]
\\
\label{eq:ExistenceContractionOp1}
\leq& 2 \bE\Big[ \int_0^T | \langle Y^{\mu}(s) - Y^{\nu}(s), b(s, Y^{\mu}(s), \mu_s) - b(s, Y^{\nu}(s), \nu_s) \rangle | ds \Big]
\\
\label{eq:ExistenceContractionOp2}
&+ 2 \bE\Big[ \sup_{t\in[0,T]} \int_0^t \big\langle Y^{\mu}(s) - Y^{\nu}(s), [\sigma(s, Y^{\mu}(s), \mu_s) - \sigma(s, Y^{\nu}(s), \nu_s)]dW(s) \big\rangle \Big]
\\
\label{eq:ExistenceContractionOp3}
&+ \bE\Big[ \int_0^T \Big| \sigma(s, Y^{\mu}(s), \mu_s) - \sigma(s, Y^{\nu}(s), \nu_s)\Big|^2 ds \Big].
\end{align}

Firstly, we apply the monotonicity and Lipschitz properties to get
\begin{align*}
\eqref{eq:ExistenceContractionOp1}\leq& 2L \bE\Big[ \int_0^T |Y^\mu(s) - Y^\nu(s)|^2 ds\Big] + 2L \bE\Big[ \int_0^T |Y^{\mu}(s) - Y^{\nu}(s)| W^{(2)}(\mu_s, \nu_s) ds \Big]
\\
\leq& 2L\int_0^T \bE\Big[ \|Y^\mu - Y^\nu\|_{\infty, s} \Big] ds + \frac{\bE\Big[ \|Y^\mu - Y^\nu\|_\infty^2 \Big]}{3} + 3L^2 \int_0^T W^{(2)}(\mu_s, \nu_s)^2 ds. 
\end{align*}

Secondly, we use the Burkholder Davis Gundy inequality and Lipschitz properties
\begin{align*}
\eqref{eq:ExistenceContractionOp2}\leq& 2 \bE\Big[ \Big( \int_0^T \Big| (Y^\mu(s) - Y^\nu(s))^T \Big( \sigma(s, Y^{\mu}(s), \mu_s) - \sigma(s, Y^{\nu}(s), \nu_s)\Big) \Big|^2 ds\Big)^{\tfrac{1}{2}} \Big]
\\
\leq& 2\bE\Big[ \|Y^\mu - Y^\nu\|_\infty \Big( \int_0^T |\sigma(s, Y^{\mu}(s), \mu_s) - \sigma(s, Y^{\nu}(s), \nu_s)|^2 ds\Big)^{\tfrac{1}{2}} \Big]
\\
\leq& \frac{\bE\Big[ \|Y^\mu - Y^\nu\|_\infty^2 \Big]}{3} + 6L^2 \int_0^T \bE\Big[ \| Y^{\mu} - Y^{\nu}\|_{\infty, s}^2\Big] ds + 6L^2 \int_0^T W^{(2)} (\mu_s, \nu_s)^2 ds. 
\end{align*}

Thirdly, using the Lipschitz properties again we get
\begin{align*}
\eqref{eq:ExistenceContractionOp2} \leq& 2L^2 \int_0^T \bE\Big[ \|Y^\mu - Y^\nu\|_{\infty, s}^2\Big] ds + 2L^2\int_0^T W^{(2)}(\mu_s, \nu_s)^2 ds. 
\end{align*}
Combining \eqref{eq:ExistenceContractionOp1}, \eqref{eq:ExistenceContractionOp2} and \eqref{eq:ExistenceContractionOp3} gives that
$$
\frac{\bE\Big[\|Y^{\mu} - Y^{\nu}\|_{\infty}^2 \Big]}{3} 
\leq (8L^2 + 2L) \int_0^T \bE\Big[\|Y^{\mu} - Y^{\nu}\|_{\infty, s}^2 \Big] ds + 11L^2 \int_0^T W^{(2)}(\mu_s, \nu_s)^2 ds. 
$$
Applying Gr\"onwall to this yields a control to the initial inequality
\begin{align*}
W^{(2)} ( \Xi(\mu), \Xi(\nu))^2 \leq& \bE\Big[\|Y^{\mu} - Y^{\nu}\|_{\infty}^2 \Big] 
\leq K \int_0^T W_s^{(2)} (\mu, \nu)^2 ds,
\end{align*}
where $K = 11L^2 e^{(24L^2+6L)T} $. Applying $\Xi$ inductively $j$ times yields
\begin{align*}
W^{(2)} \Big( \Xi^j(\mu), \Xi^j(\nu)\Big)^2 \leq& K^j \int_0^T \int_0^{t_1} ... \int_0^{t_{j-1}} W_{t_j}^{(2)} (\mu, \nu)^2 dt_1 ... dt_{j}
\\
\leq&K^j \int_0^T \frac{(T-t_j)^{j-1} }{(j-1)!}W_{t_j}^{(2)} (\mu, \nu)^2 dt_{j} \leq \frac{K^j T^j}{j!} W^{(2)} (\mu, \nu)^2.
\end{align*}
Choosing $j$ large enough ensures that $\Xi^j$ is a contraction operator. Therefore, $\Xi$ has a unique fixed point. 
Hence we conclude that the Picard sequence of random processes $Y^0 (t) = Y(0)$ and
$$
dY^n(t) = b\big(t, Y^n(t), \cL_t^{Y^{n-1}}\big) dt + \sigma\big(t, Y^n(t), \cL_t^{Y^{n-1}}\big) dW(t),
$$
converges in $\cS^2$ and the limit solves the MV-SDE \eqref{eq:MKSDE-MainExistTheo}. From this we conclude that a unique solution exists. 

\emph{Step 2: Moment calculations}. Recall the dynamics of $Y$ from \eqref{eq:MKSDE-MainExistTheo}. By 
 It\^o's formula we have 
\begin{align*}
|Y(t)|^p =& |Y(0)|^p
+ p\int_0^t |Y(s)|^{p-2} \langle Y(s), b(s, Y(s), \cL_s^{Y})\rangle ds 
\\
&+ p\int_0^t |Y(s)|^{p-2} \langle Y(s), \sigma(s, Y(s), \cL_s^{Y}) dW(s)\rangle 
\\
&
+ \frac{p}{2} \int_0^t |Y(s)|^{p-2} \Big| \sigma(s, Y(s), \cL_s^{Y})\Big|^2 ds
\\
&+ \frac{p(p-2)}{2} \int_0^t |Y(s)|^{p-4} \cdot \Big| Y(s)^T \sigma(s, Y(s), \cL_s^Y) \Big|^2 ds.
\end{align*}
Therefore
\begin{align}
\nonumber
\bE\Big[ \|Y\|_{\infty}^p\Big] =& \bE\Big[ |Y(0)|^p \Big]
\\
\label{eq:NewExistenceMoments1}
&+ p\bE\Big[ \int_0^T |Y(s)|^{p-2} \big|\langle Y(s), b(s, Y(s), \cL_s^{Y})\rangle\big| ds \Big]
\\
\label{eq:NewExistenceMoments2}
&+ p\bE\Big[ \sup_{0\leq t\leq T}\int_0^t |Y(s)|^{p-2} \langle Y(s), \sigma(s, Y(s), \cL_s^{Y})dW(s)\rangle \Big]
\\
\label{eq:NewExistenceMoments3}
&+ \frac{p}{2} \bE\Big[ \int_0^T |Y(s)|^{p-2} \Big| \sigma(s, Y(s), \cL_s^{Y})\Big|^2 ds \Big] 
\\
\label{eq:NewExistenceMoments4}
&+ \frac{p(p-2)}{2} \bE\Big[ \int_0^T |Y(s)|^{p-4} \cdot \Big| Y(s)^T \sigma(s, Y(s), \cL_s^Y) \Big|^2 ds \Big].
\end{align}
By the triangle property we have
\begin{align}
\label{eq:NewExistenceMoments1.1}
\eqref{eq:NewExistenceMoments1} \leq& p \bE\Big[ \int_0^T |Y(s)|^{p-2} \langle Y(s), b(s, Y(s), \cL_s^Y) - b(s, 0, \cL_s^Y) \rangle ds \Big] 
\\
\label{eq:NewExistenceMoments1.2}
&+p \bE\Big[ \int_0^T |Y(s)|^{p-2} \langle Y(s), b(s, 0, \cL_s^Y) - b(s, 0, \delta_0) \rangle ds \Big] 
\\
\label{eq:NewExistenceMoments1.3}
&+p \bE\Big[ \int_0^T |Y(s)|^{p-2} \langle Y(s), b(s, 0, \delta_0) \rangle ds \Big] . 
\end{align}
Using the monotone property of $b$ yields
$
\eqref{eq:NewExistenceMoments1.1} \leq pL \int_0^T \bE\Big[ \|Y\|_{\infty, s}^p\Big] ds. 
$ 
Using the Lipschitz property of $b$ in the distribution variable and Lemma \ref{lemma:WassersteinAgainstDirac} yields
\begin{align*}
\eqref{eq:NewExistenceMoments1.2} \leq& pL  \int_0^T \bE\Big[ \|Y\|_{\infty, s}^{p-1} \Big] \bE\Big[ \|Y\|_{\infty, s}^2\Big]^{\tfrac{1}{2}} ds
\\
\leq&pL \int_0^T \Big( \frac{(p-1)\bE\Big[ \|Y\|_{\infty, s}^{p-1}\Big]^{\tfrac{p}{p-1}} }{p} + \frac{\bE\Big[ \|Y\|_{\infty, s}^2\Big]^{\tfrac{p}{2}}}{p} \Big) ds
\\ \leq &  
 pL \int_0^T \bE\Big[ \|Y\|_{\infty, s}^p\Big] ds. 
\end{align*}
Using the integrability properties of $b$ yields
\begin{align*}
\eqref{eq:NewExistenceMoments1.3} 
\leq
& \bE\Big[ \|Y\|_{\infty}^{p-1} \int_0^T |b(s, 0, \delta_0)| ds\Big] 
\\
&
\leq \frac{\bE\Big[ \|Y\|_{\infty}^p \Big] }{n} + n^{p-1} (p-1)^{p-1} \bE\Big[ \Big( \int_0^T |b(s, 0, \delta)| ds \Big)^p \Big] 
\end{align*}
where $n\in \bN$ which will be chosen later. 

By the Burkholder-Davis-Gundy inequality, the Lipschitz properties and Lemma \ref{lemma:WassersteinAgainstDirac} we have
\begin{align}
\nonumber
\eqref{eq:NewExistenceMoments2}\leq& pC_1 \bE\Big[ \Big( \int_0^T |Y(s)|^{2p-4} \Big|Y(s)^T\sigma(s, Y(s), \cL_s^{Y})\Big|^2 ds \Big)^{\tfrac{1}{2}} \Big]
\\
\nonumber
\leq& pC_1 \bE\Big[ \|Y\|_{\infty, s}^{\tfrac{p}{2}} \Big( \int_0^T |Y(s)|^{p-2} \cdot \Big| \sigma(s, Y(s), \cL_s^{Y})\Big|^2 ds\Big)^{\tfrac{1}{2}} \Big]
\\
\nonumber
\leq& \frac{\bE\Big[ \|Y\|_{\infty, s}^p \Big]}{n} + p^2 C_1^2 n \bE\Big[ \int_0^T |Y(s)|^{p-2} \cdot \Big| \sigma(s, Y(s), \cL_s^Y) \Big|^2 ds \Big]
\\
\label{eq:NewExistenceMoments2.1}
\leq& \frac{\bE\Big[ \|Y\|_{\infty, s}^p \Big]}{n} + 3p^2 C_1^2 n L^2\Big( \int_0^T  \bE\Big[ \|Y\|_{\infty, s}^{p} \Big] ds + \int_0^T \bE\Big[\|Y\|_{\infty, s}^{p-2}\Big] \cdot \bE\Big[ \|Y\|_{\infty, s}^2\Big] ds \Big)
\\
\label{eq:NewExistenceMoments2.2}
& + 3p^2 C_1^2 n \bE\Big[ \int_0^T |Y(s)|^{p-2} \Big| \sigma(s, 0, \delta_0) \Big|^2 ds \Big]. 
\end{align}
Terms \eqref{eq:NewExistenceMoments2.1} are dealt with in the same way as terms \eqref{eq:NewExistenceMoments1.1} and \eqref{eq:NewExistenceMoments1.2}. For \eqref{eq:NewExistenceMoments2.2} we proceed as follows
\begin{align*}
\eqref{eq:NewExistenceMoments2.2} \leq& \bE\Big[ \|Y\|_{\infty}^{p-2} \Big( 3p^2 C_1^2 n \int_0^T \Big| \sigma(s, 0, \delta_0)\Big|^2 ds\Big) \Big] 
\\
\leq& \frac{\bE\Big[ \|Y\|_{\infty}^{p}\Big]}{n} + 2\cdot3^{\tfrac{p}{2}}\cdot n^{p-1} C_1^p (p-2)^{\tfrac{p-2}{2}} p^{\tfrac{p}{2}} \bE\Big[ \Big( \int_0^T \Big| \sigma(s, 0, \delta_0)\Big|^2 ds\Big)^{\tfrac{p}{2}} \Big] . 
\end{align*}
Thirdly, we have
\begin{align}
\nonumber
\eqref{eq:NewExistenceMoments3} + \eqref{eq:NewExistenceMoments4} \leq& \frac{p(p-1)}{2} \bE\Big[ \int_0^T |Y(s)|^{p-2} \Big| \sigma(s, Y(s), \cL_s^{Y})\Big|^2 ds \Big] 
\\
\nonumber
\leq& \frac{3p(p-1)L^2}{2} \Big( \int_0^T \bE\Big[ \|Y\|_{\infty, s}^p \Big] ds + \int_0^T \bE\Big[ \|Y\|_{\infty, s}^{p-2} \Big] \cdot \bE\Big[ \|Y\|_{\infty, s}^{2} \Big] ds \Big)
\\
\label{eq:NewExistenceMoments3.1}
&+\frac{3p(p-1)}{2} \bE\Big[ \|Y\|_{\infty}^{p-2} \int_0^T \Big| \sigma(s, 0, \delta_0)\Big|^2 ds \Big]
\end{align}
and
\begin{align*}
\eqref{eq:NewExistenceMoments3.1} \leq \frac{ \bE\Big[ \|Y\|_\infty^p\Big] }{n} + \Big( \frac{n(p-2)}{2}\Big)^{\tfrac{p-2}{2}} \cdot \Big( 3(p-1)\Big)^{\tfrac{p}{2}} \bE\Big[ \Big( \int_0^T \Big| \sigma(s, 0, \delta_0)\Big|^2 ds\Big)^{\tfrac{p}{2}} \Big]. 
\end{align*}

Hence we choose $n=5$ and this can all be rearranged to get
\begin{align*}
\frac{\bE\Big[ \|Y\|_{\infty}^p\Big]}{5} \leq& \bE\Big[ |Y(0|^p \Big] + \tilde{C_1} \bE\Big[ \Big( \int_0^T \Big| \sigma(s, 0, \delta_0)\Big|^2 ds\Big)^{\tfrac{p}{2}} \Big] 
\\
&+ \tilde{C_2} \bE\Big[ \Big( \int_0^T |b(s, 0, \delta_0)| ds\Big)^{p} \Big] 
+ \tilde{C_3} \int_0^T \bE\Big[ \|Y\|_{\infty, s}^p\Big] ds,
\end{align*}
where the constants $\tilde{C_1}$, $\tilde{C_2}$ and $\tilde{C_3}$ are dependent only on $p$ and $L$. Applying Gr\"onwall's lemma provides us with the $p$ moment upper bound 
\begin{align*}
\bE\Big[ \|Y\|_{\infty}^p\Big] 
\leq 5\Bigg( \bE\Big[ |Y(0|^p \Big] 
& + \tilde{C_1} \bE\Big[ \Big( \int_0^T \Big| \sigma(s, 0, \delta_0)\Big|^2 ds\Big)^{\tfrac{p}{2}} \Big] 
\\
&
+ \tilde{C_2} \bE\Big[ \Big( \int_0^T |b(s, 0, \delta_0)| ds\Big)^{p} \Big] \Bigg) e^{\tilde{C_3} T}. 
\end{align*}

\end{proof}

\subsection{Continuity in time behavior}

We next give results describing time-continuity for the process and its law in the appropriate topologies.
\begin{proposition}
\label{proposition:ContinuityOfLaw}
Let $Y$ be the solution of \eqref{eq:MKSDE-MainExistTheo} satisfying Assumption \ref{ass:MKSDE-MainExistTheo} where $q\in \bN$ is the order of the polynomial growth of $b$. Let $n\in \bN$ and $n\geq 2$ and additionally assume that
$$
\bE\Big[ \sup_{t\in[0,T]} |b(t, 0, \delta_0)|^{nq} \Big],\ \bE\Big[ \sup_{t\in[0,T]} \Big| \sigma(t, 0, \delta_0) \Big|^{\tfrac{nq}{2}} \Big] < \infty.
$$
Then for every $t,s\in[0,T]$  
$$
W^{(n)}\Big( \cL_t^Y, \cL_s^Y \Big)  \leq \bE\Big[ \Big|Y(t) - Y(s)\Big|^n\Big]^{\tfrac{1}{n}} \lesssim |t-s|^{\frac{1}{2}}.
$$
\end{proposition}

\begin{proof} 
The proposition's conditions mean that by the arguments in the proof of Theorem \ref{theo:LocLipMcK-V.ExistUniq} we have
$
\bE[ \sup_{t\in[0,T]} |Y(t)|^{nq}] < \infty.
$ 
Take $0\leq s \leq t \leq T<\infty$ and a natural number $n\geq 2$. We have
$$
|Y(t)-Y(s)|^n \leq \Big| \int_s^t b(r, Y(r), \cL_r^Y) dr + \int_s^t \sigma(r, Y(r), \cL_r^Y) dW(r) \Big|^n. 
$$
We use the growth condition of $b$ and the Lipschitz property of $\sigma$ and apply the Minkowski Inequality to get
\begin{align*}
\bE\Big[ |Y(t) & - Y(s)|^n\Big]^{\tfrac{1}{n}} \\
\leq& \bE\Big[ \Big|\int_s^t b(r, Y(r), \cL_r^Y)dr \Big|^n \Big]^{\tfrac{1}{n}} + \bE\Big[ \Big|\int_s^t \sigma(r,Y(r), \cL_r^Y) dW(r) \Big|^n \Big]^{\tfrac{1}{n}}
\\
\leq& \bE\Big[ \Big| \int_s^t |b(r, Y(r), \cL_r^Y) dr \Big|^n \Big]^{\tfrac{1}{n}} 
+\bE\Big[ \Big| \int_s^t |\sigma(r, Y(r), \cL_r^Y)|^2 dr \Big|^{\tfrac{n}{2}} \Big]^{\tfrac{1}{n}}
\\
\leq& |t-s| \bE\Big[ \Big( \|b(\cdot, 0, \delta_0)\|_\infty + L\|Y\|_\infty^q + \bE\Big[\|Y\|^2\Big]^{\tfrac{1}{2}} \Big)^n\Big]^{\tfrac{1}{n}}  
\\
&+ |t-s|^{\tfrac{1}{2}} \bE\Big[ \Big( \|\sigma(\cdot, 0, \delta_0)\|_\infty + L\|Y\|_\infty + \bE\Big[ \|Y\|_\infty^2\Big]^{\tfrac{1}{2}}\Big)^n \Big]^{\tfrac{1}{n}} 
\\
 \lesssim& |t-s|^{\frac{1}{2}},
\end{align*}

From the 1st part of the proposition, we have $\bE\big[\, |Y(t) - Y(s)|^{2p}\big]\lesssim |t-s|^p$. The results now follow by applying Kolmogorov's Continuity criterion in a standard fashion.
\end{proof}

\begin{corollary}
Let $Y$ be the solution of \eqref{eq:MKSDE-MainExistTheo} under Assumption \ref{ass:MKSDE-MainExistTheo} and suppose additionally that $\forall n \in \bN$ we have 
$$
\bE\Big[ \| b(\cdot, 0, \delta_0)\|_\infty^n \Big] < \infty, 
\quad
\bE\Big[ \| \sigma(\cdot, 0, \delta_0)\|_\infty^n \Big] < \infty. 
$$
Then there is a modification of $Y(\cdot)$, $\tilde{Y}(\cdot)$,  which is sample-continuous, almost surely equal to $Y(\cdot)$ and $\alpha$-H\"older continuous for $\alpha < {1}/{2}$. 
\end{corollary}

\begin{proof}
Under these stronger conditions we have  $\forall n\in \bN$ that 
$
\bE[ \, |Y(t) - Y(s)|^n] \lesssim |t-s|^{{n}/{2}}
$. 
Therefore, we apply the Kolmogorov Continuity Criterion from \cite[Theorem 2.2.3]{oksendal2003stochastic} and conclude. 
\end{proof}
The final result concerns $C^1$-regularity (in time) of the expected value of maps of the MV-SDE's solution.
\begin{proposition}[Regularity in time]
Let $\phi\in C^{1,2}([0,T]\times \bR^d)$ and suppose that $\phi$, its first derivative, $\nabla_x \phi(\cdot, \cdot)$, and Hessian, $H[\phi](\cdot,\cdot)$, have polynomial growth such that for some $r>0$ and some $K>0$
$$
\max\Big\{
\big| \frac{\partial \phi}{\partial t}(t, y)\big|, 
\big|\nabla_x \phi(t, y) \big|,
\big|H[\phi](t, y)\big|
\Big\}
\leq K\big(1+ |y|^r\big).
$$
Suppose that $Y$ is the solution to \eqref{eq:MKSDE-MainExistTheo} under Assumptions \ref{ass:MKSDE-MainExistTheo} with $p:= \max\{r+q,2r+2\}$ ($q$ is the polynomial growth of $b$) and hence $Y\in \cS^{p}$. 

Then $t\mapsto \bE[\phi(t, Y(t))]\in C^1$ and 
\begin{align*}
\partial_t \bE[\phi(t, Y(t))]
=& \bE\Big[ \frac{\partial \phi}{\partial t}(t, Y(t)) \Big] + \bE\Big[ \nabla \phi(t, Y(t))^T \cdot b(t, Y(t), \cL_t^Y) \Big] 
\\
&+ \bE\Big[ \mbox{Tr}\Big(\sigma(t, Y(t), \cL_t^Y)^T\cdot H\Big[\phi \Big](t, Y(t)) \cdot \sigma(t, Y(t), \cL_t^Y) \Big) \Big].
\end{align*}
\end{proposition}
\begin{proof}
Use It\^o's formula on $\phi(t,Y(t))$, integrate over $[0,t]$ and take expectations. By the integrability/growth assumptions on $b$ and $\sigma$, we have $Y\in\cS^{p}$ and in particular $Y\in\cS^{2r+2}$. Combining with the polynomial growth of $\nabla \phi(\cdot,\cdot)$ in its spatial variable we easily conclude that the stochastic integral $\int_0^\cdot \nabla \phi(s,Y(s)) \sigma(s, Y(s), \cL_s^Y) ) d W(s)$ is a square-integrable martingale and hence it vanishes under the expectation.

In the previous results we have shown continuity in time of $Y$ and $\cL^Y$ in the appropriate metrics. This, combined with the continuity of $b$ and $\sigma$ in their variables plus the integrability results, allows to apply Fubini and swap expectations and integrals. Lastly, using the continuity/integrability properties of the involved terms again (notice that here one requires $Y\in \cS^{r+q}$), one can compute the time derivative of $t\mapsto \bE[\phi(t, Y(t))]$ via the Leibniz differentiation rule for integrals. This yields the lemma's formula. 
\end{proof}


\section{Large Deviations Principle}
\label{sec:LPDresults}

In this section we investigate the family of $d$-dimensional MV-SDEs indexed to a parameter $\varepsilon>0$, 
\begin{align}
\label{sandra1}
X_{\varepsilon}^x(t) 
&= x
   + \int_0^t b_{\varepsilon}(s, X_{\varepsilon}^x(s), \cL_s^{X_{\varepsilon}^x})ds 
	 + {\sqrt{\varepsilon}} \int_0^t \sigma_{\varepsilon}(s, X_{\varepsilon}^x(s), \cL_s^{X_{\varepsilon}^x}) dW(s).
\end{align}
We derive two types of LDP for the above SDE. The first is an LDP for the supremum norm while the second is an LDP for the H\"older-norm. Throughout we make use of several known sources: \cite{DemboZeitouni2010}, \cite{HerrmannImkellerPeithmann2008} and \cite{arous1994grandes}. The main contribution of this section apart from the LDPs themselves, are the techniques needed to deal directly with the law of the solution process inside the coefficients avoiding measure arguments; time dependency of the coefficients is included. For technical convenience, we work on the time interval $[0,1]$. The extension to the interval $[0,T]$ is straightforward.

\begin{assumption}
\label{def:He-man+Skeletor}
Let $\varepsilon>0$. Let $b, b_{\varepsilon} :[0,1] \times  \bR^d \times\cP_2(\bR^d) \to \bR^d$, $\sigma, \sigma_{\varepsilon} :[0,1] \times  \bR^d \times \cP_2(\bR^d) \to \bR^{d\times d'}$ (deterministic maps) and $x\in\bR^d$.

As $\varepsilon\searrow 0$, let the maps $b_{\varepsilon}$ converge uniformly to $b$ and $\sigma_{\varepsilon}$ converge uniformly to $\sigma$. Let $b$ and $\sigma$ satisfy Assumption \ref{ass:MKSDE-MainExistTheo} with the additional restrictions that there exists $M>0$ such that $\sigma$ is bounded by $M$ and that there exists $\beta \in (0,1]$ such that for any $s,s'\in[0,1]$, for any $y\in\mathbb{R}^d$ and for all $\mu\in\mathcal{P}_2\left(\mathbb{R}^d\right)$, we have:
\begin{align*}
\left|\sigma\left(s,y,\mu\right)-\sigma(s',y,\mu)\right|\leq L|s-s'|^\beta, \quad 
\left|b\left(s,y,\mu\right)-b(s',y,\mu)\right|\leq L|s-s'|^\beta\,.
\end{align*}
\end{assumption}

\begin{remark}
For this section we only worry about the conditions on the coefficients $b$ and $\sigma$. However, we will additionally assume that $b_{\varepsilon}$ and $\sigma_{\varepsilon}$ have adequate conditions to ensure the existence and uniqueness of a solution to the McKean-Vlasov SDE. 
\end{remark}


\subsection{Large Deviations Principle with the supremum norm}

To study \eqref{sandra1} and establish an LDP in the supremum norm for \eqref{sandra1} we will need to consider several approximations for it. We start by considering the following ordinary differential equation: 
\begin{equation}
\label{ode}
\dot\psi(t)=b\left(t,\psi(t),\delta_{\psi(t)}\right),\quad \psi(0)=x.
\end{equation}
Indeed, informally, when ${\varepsilon} \searrow 0$ in \eqref{sandra1}, the diffusion term vanishes and we have
\begin{equation*}
X_{0}^x(t)=x+\int_0^tb\left(s,X_0^x(s),\cL^{X_0^x}_s\right)ds\,.
\end{equation*}
Of course, since $x$ is deterministic, we deduce that $\cL^{X_0^x}_\cdot$ is a Dirac measure centered on the path $X^x_0(\cdot)$. Thus, the ordinary differential equation \eqref{ode} is, from a heuristically standpoint, a good approximation of the stochastic differential equation \eqref{sandra1} as $\varepsilon$ is small. Moreover, the law of $X^X_\varepsilon(t)$ can be approximated by $\delta_{\psi(t)}$. We thus define the following equation (which is closer to a standard SDE)
\begin{equation}
\label{sandra2}
Y_\varepsilon^x(t)
=x
+\int_0^tb\left(s,Y_\varepsilon^x(s),\delta_{\psi(s)}\right)ds
+{\sqrt{\varepsilon}}\int_0^t\sigma\left(s,Y_\varepsilon^x(s),\delta_{\psi(s)}\right)dW(s)
\,.
\end{equation}
However, \eqref{sandra2} has a diffusion coefficient which is not constant. As a consequence, we need to discretize:
\begin{align}
\label{sandra3}
\nonumber
Y_{\varepsilon,m}^x(t)
=x
+\int_0^tb\Big( &\frac{\lfloor ms\rfloor}{m}, Y_\varepsilon^x\left(\frac{\lfloor ms\rfloor}{m}\right),\delta_{\psi\left(\frac{\lfloor ms\rfloor}{m}\right)}\Big)ds
\\
&+{\sqrt{\varepsilon}}\int_0^t
\sigma\left(\frac{\lfloor ms\rfloor}{m}
            ,Y_\varepsilon^x\left(\frac{\lfloor ms\rfloor}{m}\right)
						,\delta_{\psi\left(\frac{\lfloor ms\rfloor}{m}\right)}\right)dW(s),
\end{align}
where $m\in \bN$ and will go to infinity. Here, $\lfloor ms\rfloor$ stands for the floor of $ms$. Lastly, we state a simple result concerning the solvability of \eqref{ode}
\begin{lemma}
\label{BloodyODE}
Under Assumption \ref{def:He-man+Skeletor}, there exists a unique solution $\psi\in C([0,1])$ to \eqref{ode}. Moreover, the map $t\to \psi(t)$ is Lipschitz continuous.
\end{lemma}

\begin{proof}
Existence of a local solution comes from the Peano Existence Theorem. Uniqueness follows from the Monotonicity and Lipschitz/locally Lipschitz properties. In order to get a global solution, we consider the square of the solution, use the monotonicity condition and Gr\"onwall to argue
\begin{align*}
\sup_{t\in[0,T]} |\psi(t)| 
=& |x| + 2 \sup_{t\in[0,T]} \int_0^t \Big\langle \psi(s), b(s, \psi(s), \delta_{\psi(s)}) \Big\rangle ds
\\
\leq&
|x| + \int_0^T \Big( 2L^2\|\psi\|_{\infty, s} + |b(s, 0, \delta_0)| \Big) ds
\\
\leq& \Big( |x| + \int_0^T |b(s, 0, \delta_{0})| ds \Big)e^{2LT}.
\end{align*}
From this we see
\begin{align*}
|\psi(t) - \psi(s)| 
&
\leq
 \int_s^t \Big( \|b(\cdot, 0, \delta_0)\|_{\infty} + L(1+\|\psi\|_{\infty}^q) + L \|\psi\|_\infty\Big) dr
\\
&
\leq
 \textrm{O}\big(|t-s|\big),
\end{align*}
which yields the Lipschitz continuity. 
\end{proof}

\subsubsection{The main result}
\label{sec:methodologyForSupSandra}
We now state the main theorem concerning an LDP for \eqref{sandra1} in the topology of the uniform norm and prove it in the remaining subsections.
\begin{theorem}
\label{samsung}
Under the hypotheses of the section, the diffusion $X^X_{\varepsilon}$ satisfies a Large Deviations Principle in $C([0,1])$ equipped with the topology of the uniform norm with the good rate function
\begin{equation*}
I^x(f) 
:= 
\inf
     \frac{1}{2}\int_0^t\left|\dot g(t)\right|^2dt,
\end{equation*}
the infimum is taken over the set 
$$\left\{g\in H^{\otimes d'}:f(t)=x+\int_0^t b\left(s,f(s),\delta_{\psi(s)}\right) ds
                         +\int_0^t \sigma\left(s,f(s),\delta_{\psi(s)}\right)\dot g(s)ds\right\},$$
and where $\psi$ is the solution to \eqref{ode}.
\end{theorem}
The ODE appearing in the infimum is easily recognizable as the Skeleton of SDE \eqref{sandra2}. Lastly, if $\sigma$ is a square matrix and if $a:=\sigma\sigma^T$ is uniformly strictly positive, the preceding formula for the rate function simplifies into
\begin{align*}
I^x(\varphi)
:=
\frac{1}{2}\int_0^t\left(\dot \varphi(t) \right.&\left. -b\left(t,\varphi(t),\delta_{\psi(t)}\right)\right)^T
\\
& \qquad \times 
a^{-1}\left(t,f(t),\delta_{\psi(t)}\right)\left(\dot \varphi(t)-b\left(t,\varphi(t),\delta_{\psi(t)}\right)\right)dt\,.
\end{align*}

\subsubsection*{Methodology}

From a methodological point of view, to show that the family $\left(X^X_{\varepsilon}\right)_{\varepsilon>0}$ satisfies a LDP in the supremum norm topology, we first show that the approximation $\left(Y_{\varepsilon,m}^x\right)_{\varepsilon>0,m\in\mathbb{N}}$ given in \eqref{sandra3} satisfies a LDP with the good rate function $I^x$ (defined below in \eqref{lavie}). Then, we prove that $\left(Y_{\varepsilon,m}^x\right)_{\varepsilon>0,m\in\mathbb{N}}$ is exponentially equivalent to $\left(Y_\varepsilon^x\right)_{\varepsilon>0}$ as $m$ goes to infinity and $\varepsilon$ goes to zero and since LDPs do not distinguish between exponentially equivalent families (see e.g. \cite[Theorem 2.21]{herrmann2013stochastic}), we deduce that $\left(Y_\varepsilon^x\right)_{\varepsilon>0}$ satisfies a LDP with the good rate function $I^x$. Finally, we show that $\left(X_{\varepsilon}^x\right)_{\varepsilon>0}$ and $\left(Y_\varepsilon^x\right)_{\varepsilon>0}$ are exponentially equivalent as $\varepsilon$ goes to zero. This implies, via the same argument, that $\left(X_{\varepsilon}^x \right)_{\varepsilon>0}$ satisfies a LDP with the good rate function~$I^x$. We make use of standard results from \cite{DemboZeitouni2010}, some of which are recalled in the Appendix below.

\subsubsection[Large Deviations principle for ODE Decoupled MV-SDE]{Large Deviations principle for $Y_\varepsilon^x$}
We follow \cite{DemboZeitouni2010} plus the techniques used in \cite{HerrmannImkellerPeithmann2008} for having a drift coefficient which is only \emph{locally} Lipschitz, but adequately adapted to the current setting MV-SDE setting.

\begin{proposition}
\label{manue}
Under the hypotheses of the section, the family of diffusions $\left(Y^X_{\varepsilon}\right)_{\varepsilon>0}$ satisfies a Large Deviations principle in $C([0,1])$ equipped with the topology of the uniform norm with the good rate function
\begin{equation}
\label{lavie}
I^x(f):=\inf_{\left\{g\in H^{\otimes d'}\,:\,f(t)=x+\int_0^t b\left(s, f(s),\delta_{\psi(s)}\right)ds+\int_0^t \sigma\left(s, f(s),\delta_{\psi(s)}\right)\dot g(s)ds\right\}}\frac{1}{2}\int_0^1\left|\dot g(t)\right|^2dt\,,
\end{equation}
$\psi$ is the unique solution to \eqref{ode}.
\end{proposition}
Before proving Proposition \ref{manue}, we first show that the approximation $\left(Y_{\varepsilon,m}\right)_{\varepsilon>0,m\in\mathbb{N}}$ satisfies a Large Deviation Principle with the good rate function $I^x_m$ defined as
\begin{equation*}
I^x_m(f):=
\inf
\frac{1}{2}\int_0^1\left|\dot g(t)\right|^2dt,
\end{equation*}
where the infimum is taken over the set
\begin{align*}
\Bigg\{g\in H^{\otimes d'}\,:\,
f(t)=x
&+\int_0^t b\left(\frac{\lfloor ms\rfloor}{m}, f\big(\frac{\lfloor ms\rfloor}{m}\big), \delta_{\psi\big(\frac{\lfloor ms\rfloor}{m}\big)}\right)ds 
\\ & 
 +\int_0^t\sigma\left(\frac{\lfloor ms\rfloor}{m},f\big(\frac{\lfloor ms\rfloor}{m}\big)
                                                 ,\delta_{\psi\big(\frac{\lfloor ms\rfloor}{m}\big)}\right)\dot g(s)ds\Bigg\}.
\end{align*}
This is an easy exercise using the contraction principle (see \cite{DemboZeitouni2010}) so the proof is omitted. Let us just note that we need to introduce the map $F^m$ defined via $h=F^m(g)$, where
\begin{align*}
h(t) = h\left(\frac{k}{m}\right)
     &+ b\left(\frac{k}{m},h\left(\frac{k}{m}\right),\delta_{\psi\left(\frac{k}{m}\right)}\right)\left(t-\frac{k}{m}\right)\\
		 & +\sigma\left(\frac{k}{m},h\left(\frac{k}{m}\right),\delta_{\psi\left(\frac{k}{m}\right)}\right)\left(g(t)-g\left(\frac{k}{m}\right)\right)\,,
\end{align*}
for $t\in\left[\frac{k}{m}, \frac{k+1}{m}\right]$, $0\leq k\leq m-1$, and $h(0)=x$.

Proposition \ref{manue} follows by showing exponential equivalence as $\varepsilon$ goes to $0$ and $m$ goes to infinity of the the families $\left(Y^X_{\varepsilon}\right)_{\varepsilon>0}$ and $\left(Y^x_{\varepsilon,m}\right)_{\varepsilon>0,m\in\bN}$.
\begin{lemma}
\label{scooter}
For any $\delta>0$, we have:
\begin{equation*}
\lim_{m\to+\infty} \limsup_{\varepsilon\to 0}
\varepsilon\log\Big( \mathbb{P}\Big[ \|Y_\varepsilon^x-Y_{\varepsilon,m}^x\|_\infty>\delta \Big] \Big)=-\infty\,.
\end{equation*}
\end{lemma}
\begin{proof}
Fix $\delta>0$. Let $z_t:=Y^x_{\varepsilon,m}(t)-Y^X_\varepsilon(t)$, and for any $\rho,R>0$, define the stopping time 
\begin{equation*}
\tau:=\min\Big\{ \inf\left\{t\geq0:
\left|Y^x_{\varepsilon,m}(t)\right|\geq R+1\right\}\, , \,
                  \inf\left\{t\geq0:\left|Y^X_{\varepsilon}(t)\right|\geq R+1\right\}
				  \Big\},
\end{equation*}
then
\begin{equation*}
\tau_1:=\min\left\{ 1 \ , \ \inf\big\{t\geq0:\big|Y_\varepsilon^x\Big(\frac{\lfloor mt\rfloor}{m}\Big)-Y_\varepsilon^x(t)\big|\geq\rho\big\}\ , \ \tau \right\}.
\end{equation*}
The process $\left(z_t\right)_{t\in[0,1]}$ is of the form \eqref{truc}, with $z_0=0$,
\begin{align*}
b_t
&:=b\left(\frac{\lfloor mt\rfloor}{m},Y_\varepsilon^x \left(\frac{\lfloor mt\rfloor}{m}\right),\delta_{\psi\left(\frac{\lfloor mt\rfloor}{m}\right)}\right)-b\big(t,Y_\varepsilon^x(t),\delta_{\psi(t)}\big),
\\
\sigma_t
&:=\sigma\left(\frac{\lfloor mt\rfloor}{m},Y_\varepsilon^x \left(\frac{\lfloor mt\rfloor}{m}\right),\delta_{\psi\left(\frac{\lfloor mt\rfloor}{m}\right)}\right)-\sigma\big(t,Y_\varepsilon^x(t),\delta_{\psi(t)}\big)\,.
\end{align*}
Thus, by the local Lipschitz continuity of $b$, by the global Lipschitz continuity of $\sigma$ and the definition of $\tau_1$, it follows that Lemma \ref{eurodance} is applicable here. Indeed, we have:
\begin{align*}
\left|\sigma_t\right|&=\left|\sigma\left(\frac{\lfloor mt\rfloor}{m},Y^X_{\varepsilon}\left(\frac{\lfloor mt\rfloor}{m}\right),\delta_{\psi\left(\frac{\lfloor mt\rfloor}{m}\right)}\right)-\sigma\left(t,Y^X_\varepsilon(t),\delta_{\psi(t)}\right)\right|\\
&\leq\left|\sigma\left(\frac{\lfloor mt\rfloor}{m},Y^X_{\varepsilon}\left(\frac{\lfloor mt\rfloor}{m}\right),\delta_{\psi\left(\frac{\lfloor mt\rfloor}{m}\right)}\right)-\sigma\left(\frac{\lfloor mt\rfloor}{m},Y^X_{\varepsilon}\left(\frac{\lfloor mt\rfloor}{m}\right),\delta_{\psi(t)}\right)\right|\\
&\qquad \qquad
+\left|\sigma\left(\frac{\lfloor mt\rfloor}{m},Y^X_{\varepsilon}\left(\frac{\lfloor mt\rfloor}{m}\right),\delta_{\psi(t)}\right)-\sigma\left(\frac{\lfloor mt\rfloor}{m},Y^X_\varepsilon(t),\delta_{\psi(t)}\right)\right|\\
&\qquad \qquad
+\left|\sigma\left(\frac{\lfloor mt\rfloor}{m},Y^X_\varepsilon(t),\delta_{\psi(t)}\right)-\sigma\left(t,Y^X_\varepsilon(t),\delta_{\psi(t)}\right)\right|\\
&\leq L\left|\psi\left(\frac{\lfloor mt\rfloor}{m}\right)-\psi(t)\right|+\frac{L}{m^\beta}+L|z_t|\\
&\leq M\left(\rho(m)^2+|z_t|^2\right)^{\frac{1}{2}}\,,
\end{align*}
with $M$ large enough and $\rho(m):=\max\left\{\sup_{t\in[0,1]}\left|\psi\left(\frac{\lfloor mt\rfloor}{m}\right)-\psi(t)\right|,\frac{1}{m^\beta}\right\}$ which, using the continuity of $\psi$, goes to $0$ as $m$ goes to infinity.

We argue as follows for the drift coefficient $|b_t|$,
\begin{align*}
\left|b_t\right|
&
=\left|b\left(\frac{\lfloor mt\rfloor}{m},Y^X_{\varepsilon}\left(\frac{\lfloor mt\rfloor}{m}\right),\delta_{\psi\left(\frac{\lfloor mt\rfloor}{m}\right)}\right)-b\left(t,Y^X_\varepsilon(t),\delta_{\psi(t)}\right)\right|
\\
&
\leq\left|b\left(\frac{\lfloor mt\rfloor}{m},Y^X_{\varepsilon}\left(\frac{\lfloor mt\rfloor}{m}\right),\delta_{\psi\left(\frac{\lfloor mt\rfloor}{m}\right)}\right)-b\left(\frac{\lfloor mt\rfloor}{m},Y^X_{\varepsilon}\left(\frac{\lfloor mt\rfloor}{m}\right),\delta_{\psi(t)}\right)\right|
\\
&\qquad \qquad
+\left|b\left(\frac{\lfloor mt\rfloor}{m},Y^X_{\varepsilon}\left(\frac{\lfloor mt\rfloor}{m}\right),\delta_{\psi(t)}\right)-b\left(\frac{\lfloor mt\rfloor}{m},Y^X_\varepsilon(t),\delta_{\psi(t)}\right)\right|
\\
&
\qquad \qquad
+\left|b\left(\frac{\lfloor mt\rfloor}{m},Y^X_\varepsilon(t),\delta_{\psi(t)}\right)-b\left(t,Y^X_\varepsilon(t),\delta_{\psi(t)}\right)\right|
\\
&
\leq L\left|\psi\left(\frac{\lfloor mt\rfloor}{m}\right)-\psi(t)\right|+\frac{L}{m^\beta}+L_{R+1}|z_t|
\\
&
\leq B_R\left(\rho(m)^2+|z_t|^2\right)^{\frac{1}{2}}\,,
\end{align*}
with $B_R$ large enough. This yields for any $\delta>0$ and any $0<\varepsilon\leq1$,
\begin{equation*}
\varepsilon\log\left( \mathbb{P}\left[\sup_{t\in[0,\tau_1]}\left|Y^x_{\varepsilon,m}(t)-Y^X_\varepsilon(t)\right|\geq\delta\right]\right)\leq K_R+\log\left(\frac{\rho(m)^2}{\rho(m)^2+\delta^2}\right).
\end{equation*}
Hence, by considering first $\varepsilon\to 0$ and then $m\to +\infty$,
\begin{equation*}
\lim_{m\to+\infty}\limsup_{\varepsilon\to0^+}\varepsilon\log\left( \mathbb{P}\left[\sup_{t\in[0,\tau_1]}\left|Y^x_{\varepsilon,m}(t)-Y^X_\varepsilon(t)\right|\geq\delta\right]\right) =-\infty\,.
\end{equation*}
Now, since
\begin{equation}
\label{tada}
\left\{\|Y^x_{\varepsilon,m}-Y^X_{\varepsilon}\|_\infty>\delta\right\}\subset\{\tau_1<1\}
\bigcup
\Big\{\sup_{t\in[0,\tau_1]}\left|Y^x_{\varepsilon,m}(t)-Y^X_\varepsilon(t)\right|\geq\delta\Big\}\,,
\end{equation}
the lemma is proved as soon as we show that for all $\rho>0$ and for any $R>0$,
\begin{equation*}
\lim_{m\to\infty}\limsup_{\varepsilon\to0}\varepsilon\log\left(\mathbb{P}\left[\tau_1<1\right]\right)=-\infty\,.
\end{equation*}
To this end, observe first that for $t\leq\tau_1$:
\begin{equation*}
\left|Y^X_{\varepsilon}\left(\frac{\lfloor mt\rfloor}{m}\right)-Y^X_\varepsilon(t)\right|\leq 
C_R\left[\frac{1}{m}+{\sqrt{\varepsilon}}\max_{0\leq k\leq m-1}\sup_{0\leq s\leq\frac{1}{m}}\left|W({\frac{k}{m}+s})-W({\frac{k}{m}})\right|\right]\,,
\end{equation*}
where $C_R$ is the maximum between the uniform bound of $\sigma$, the local bound (on the ball of center zero and radius $R+1$) of $b$ and the uniform bound of $b$ and $\sigma$ for the measure variable. Therefore, for all $m>{4C(R)}/{\rho}$,
\begin{align*}
\mathbb{P}\left[\sup_{0\leq t\leq \tau_1}\left|Y^X_{\varepsilon}\left(\frac{\lfloor mt\rfloor}{m}\right)-Y^X_\varepsilon(t)\right|\geq\frac{\rho}{2}\right]
&
\leq m\mathbb{P}\left[\sup_{0\leq s\leq\frac{1}{m}}\left|W(s)\right|\geq\frac{\frac{\rho}{2}-\frac{C_R}{m}}{{\sqrt{\varepsilon}} C_R}\right]
\\
&
\leq4dm\exp\left(-m\frac{\left(\frac{\rho}{2}-\frac{C_R}{m}\right)^2}{2d\varepsilon C_R^2}\right)\,,
\end{align*}
where the second inequality is the bound of Lemma \ref{camion}.

By taking $\delta$ sufficiently small, if $Y^X_{\varepsilon}$ exits the ball of center $0$ and of radius $R+1$, then, with high probability (quantified by the limit \eqref{tada}), the process $Y^x_{\varepsilon,m}$ exits the ball of center $0$ and of radius $R$. Consequently, to close the proof it is sufficient to prove that the probability that $Y^x_{\varepsilon,m}$ leaves the ball of center $0$ and radius $R$ is very small as $\varepsilon$ goes to zero. 

Recall, that $\left(Y^x_{\varepsilon,m}\right)_{\varepsilon>0}$ satisfies a large deviations principle with the good rate function $I^x_m$ defined previously and hence we can quantify the probability of exiting from aforementioned ball. We remark that the infimum of $I^x_m$ on the set of paths exiting from the ball of center $0$ and radius $R$ goes to infinity as $R$ goes to infinity provided that $m$ is sufficiently large. This remark can be obtained as follows. We consider $f_0:=F^m(0)$ and $f$ which is a path starting from $x$ and exiting from the ball of center $0$ and radius $R$. By $g$, we denote the function such that $f=F^m(g)$. We dominate $\left|f(t)-f_0(t)\right|$ as follows:
\begin{align*}
\left|f(t)-f_0(t)\right|&\lesssim \int_0^t\left(1+\left|f(s)\right|^q+\left|f_0(s)\right|^q\right)\left|f(s)-f_0(s)\right|\left|\dot g(s)\right|ds\\
&+\int_0^t\left|f(s)-f_0(s)\right|\left|\dot g(s)\right|ds\\
&+\int_0^t\left|\sigma\left(\frac{\lfloor ms\rfloor}{m},f_0\big(\frac{\lfloor ms\rfloor}{m}\big),\delta_{\psi\big(\frac{\lfloor ms\rfloor}{m}\big)}\right)\right|\left|\dot g(s)\right|ds\,,
\end{align*}
by using the properties on $b$ (locally Lipschitz with polynomial growth) and $\sigma$ (uniformly Lipschitz). However, the last quantity in the integral can be bounded by $C\left|\dot g(s)\right|$ where $C$ only depends on the function $f_0$. In the same vein, we obtain
\begin{align*}
\left|f(t)-f_0(t)\right|&\lesssim \int_0^t\left(1+\left|f(s)-f_0(s)\right|^q\right)\left|f(s)-f_0(s)\right|\left|\dot g(s)\right|ds+\int_0^t\left|\dot g(s)\right|ds\,.
\end{align*}
As $ab\leq\frac{1}{2}a^2+\frac{1}{2}b^2$, we get
\begin{align*}
\left|f(t)-f_0(t)\right|&\lesssim \int_0^t\left(1+\left|f(s)-f_0(s)\right|^{2q}\right)\left|f(s)-f_0(s)\right|^2ds +\int_0^t\left|\dot g(s)\right|^2ds\,,
\end{align*}
since we work on a finite time interval. However, if $\left|\left|f\right|\right|_\infty\geq R$, then we have that $\left|\left|f-f_0\right|\right|_\infty\geq \xi(R)$ with $\xi(+\infty)=+\infty$. A Gr\"onwall argument suffices to prove that it implies $\int_0^t\left|\dot g(s)\right|^2ds\geq\zeta(R)$ with $\zeta(+\infty)=+\infty$.


\end{proof}
We now are able to prove Proposition \ref{manue}.
\begin{proof}[Proof of Proposition \ref{manue}]
Let $F$ be defined on the space $H^{\otimes d'}$ such that $f=F(g)$ is the unique solution of the integral equation 
\begin{equation*}
f(t)=x+\int_0^tb\left(s,f(s),\delta_{\psi(s)}\right)ds+\int_0^t\sigma\left(s,f(s),\delta_{\psi(s)}\right)\dot g(s)ds\,.
\end{equation*}
The existence and the uniqueness of a continuous  solution is a consequence of the assumptions on $b$ and $\sigma$ and is standard.
\footnote{\label{ExistenceUniquenessFootnote}Local existence of a solution to this ODE comes from Carath\'eodory's Existence Theorem. Uniqueness comes from the Monotonicity and Lipschitz/locally Lipschitz properties. Finally, global existence comes from considering the square of the solution and using the monotonicity condition to obtain a linear growth upper bound condition which ensures the solution does not explode. The function $f$ is $1/2$-H\"older continuous. } 
In view of Lemma \ref{scooter}, the proof of the theorem is completed by combining Schilder's theorem and Proposition \ref{blablabla} 
, as soon as we show that for every $\alpha<\infty$,
\begin{equation}
\label{marina}
\lim_{m\to\infty}\sup_{g\,:\,\left\|g\right\|_{H^{\otimes d'}}\leq\alpha}\,\left\|F^m(g)-F(g)\right\|_\infty=0\,.
\end{equation}
To this end, fix $\alpha<\infty$ and $g\in H^{\otimes d'}$ such that $\left|\left|g\right|\right|_{H^{\otimes d'}}\leq\alpha$. Let $h=F^m(g)$, $f=F(g)$, and $e(t):=\left|f(t)-h(t)\right|^2$. Then for all $t\in[0,1]$,
\begin{align*}
h(t)=x& +
\int_0^tb\left(\frac{\lfloor ms\rfloor}{m},h\left(\frac{\lfloor ms\rfloor}{m}\right),\delta_{\psi\left(\frac{\lfloor ms\rfloor}{m}\right)}\right)ds\\
& +\int_0^t\sigma\left(\frac{\lfloor ms\rfloor}{m},h\left(\frac{\lfloor ms\rfloor}{m}\right),\delta_{\psi\left(\frac{\lfloor ms\rfloor}{m}\right)}\right)\dot g(s)ds\,.
\end{align*}
By the Cauchy-Schwarz inequality and the local Lipschitz property on $b$ and the global Lipschitz property on $\sigma$,
\begin{equation}
\label{marina2}
\sup_{0\leq t\leq1}\left|h(t)-h\left(\frac{\lfloor mt\rfloor}{m}\right)\right|\leq\left(\alpha+1\right)L_\alpha\delta(m)\,,
\end{equation}
where $\delta(m)$ is independent of $g$ for any $m$, and converges to zero as $m$ goes to infinity. To prove the existence of the constant $L_\alpha$, we remark that $\{g\,:\,\left\|g\right\|_{H^{\otimes d'}}\leq\alpha\}$ is a compact.

Applying the Cauchy-Schwarz inequality again, it follows by the $\beta$-H\"older time continuity of $b,\sigma$, the Lipschitz and local Lipschitz continuity of $b$ and the global Lipschitz continuity of $\sigma$ that
\begin{align*}
\left|f(t)-h(t)\right|
& \leq 
L_\alpha(\alpha+1)\sqrt{\int_0^t\left|f(s)-h\left(\frac{\lfloor ms\rfloor}{m}\right)\right|^2ds}
\\
&
\qquad\qquad\quad 
+L\sqrt{\int_0^t\left|\psi(s)-\psi\left(\frac{\lfloor ms\rfloor}{m}\right)\right|^2 ds}
+\frac{L_\alpha(\alpha+1)+L}{m^\beta}\,.
\end{align*}
Thus, due to \eqref{marina2} and the continuity of $\psi$,
\begin{equation*}
\left|f(t)-h(t)\right|^2=e(t)\leq K_\alpha\int_0^t e(s)ds+K_\alpha\delta(m)\, ,
\end{equation*}
with $e(0)=K_\alpha\delta(m)$. Hence, by Gr\"onwall's lemma, $e(t)\leq K_\alpha\delta(m)^2e^{K_\alpha t}$ and consequently 
\begin{equation*}
\left\|F(g)-F^m(g)\right\|_\infty\leq\sqrt{K_\alpha}\delta(m)e^{\frac{K_\alpha}{2}}\,,
\end{equation*}
which establishes \eqref{marina} and completes the proof.
\end{proof}

\subsubsection[Exponential equivalence]{$\left(Y^X_{\varepsilon}\right)_{\varepsilon>0}$ and $\left(X^X_{\varepsilon}\right)_{\varepsilon>0}$ are exponentially equivalent}

In order to show that $\left(X^X_{\varepsilon}\right)_{\varepsilon>0}$ satisfies a Large Deviations Principle for the uniform norm with the good rate function $I^x$, it now is sufficient to prove that the two families of processes $\left(X^X_{\varepsilon}\right)_{\varepsilon>0}$ and $\left(Y^X_{\varepsilon}\right)_{\varepsilon>0}$ are exponentially equivalent. 
\begin{proposition}
For any $\delta>0$, we have
\begin{equation*}
\limsup_{\varepsilon\to0}\varepsilon\log\left(\mathbb{P}\left[\sup_{0\leq t\leq 1}\left|X^X_\varepsilon(t)-Y^X_\varepsilon(t)\right|\geq\delta\right]\right)=-\infty\,.
\end{equation*}
\end{proposition}

\begin{proof}
The proof is similar to the one of Proposition \ref{scooter} and is also inspired by the proof of \cite[Theorem 3.4]{HerrmannImkellerPeithmann2008}.

Without loss of generality, we may choose $R>0$ such that $x$ is in the ball of center $0$ and radius $R+1$. We also assume that $\psi(t)$ does not leave this ball up to time $2$. By $\sigma'_R$, we denote the first time at which $X^X_{\varepsilon}$ or $Y^X_{\varepsilon}$ exits from the ball, then we put $\sigma_R:=\min\left\{1,\sigma'_R\right\}$. We consider $z_t:=X^X_\varepsilon(t)-Y^X_\varepsilon(t)$. This new process satisfies the following equation:
\begin{equation*}
z_t=\int_0^tb_sds+{\sqrt{\varepsilon}}\int_0^t\sigma_sdW (s)\,,
\end{equation*}
with  
\begin{align*}
b_t&:=b\left(t,X^X_\varepsilon(t),\cL_t^{X^X_{\varepsilon}}\right)-b\left(t,Y^X_\varepsilon(t),\delta_{\psi(t)}\right),
\\
\sigma_t&:=\sigma\left(t,X^X_\varepsilon(t),\cL_t^{X^X_{\varepsilon}}\right)-\sigma\left(t,Y^X_\varepsilon(t),\delta_{\psi(t)}\right)\,.
\end{align*}
Both $b_t$ and $\sigma_t$ are progressively measurable processes. We now assume that $t\leq\sigma_R$. Then, $b$ and $\sigma$ are Lipschitz in the spatial variable:
\begin{align*}
\left|b_t\right|&=\left|b\left(t,X^X_\varepsilon(t),\cL_t^{X^X_{\varepsilon}}\right)-b\left(t,Y^X_\varepsilon(t),\delta_{\psi(t)}\right)\right|\\
&\leq\left|b\left(t,X^X_\varepsilon(t),\cL_t^{X^X_{\varepsilon}}\right)-b\left(t,X^X_\varepsilon(t),\delta_{\psi(t)}\right)\right|
\\ & \qquad \qquad +\left|b\left(t,X^X_\varepsilon(t),\delta_{\psi(t)}\right)-b\left(t,Y^X_\varepsilon(t),\delta_{\psi(t)}\right)\right|\\
&\leq L\sqrt{\mathbb{E}\left[\left|X^X_\varepsilon(t)-\psi(t)\right|^2\right]}
+L_R\left|X^X_\varepsilon(t)-Y^X_\varepsilon(t)\right|\\
&\leq B_R\sqrt{\rho(\varepsilon)^2+\left|z_t^2\right|}\,,
\end{align*}
where $B_R$ depends only on $R$ and $\rho(\varepsilon):=\sup_{t\in[0,T]}\mathbb{E}\left\{\left|X^X_\varepsilon(t)-\psi(t)\right|^2\right\}$ goes to $0$ as $\varepsilon$ goes to $0$. Indeed, we can proceed as in \cite[Lemma 3.1]{HerrmannImkellerPeithmann2008} to show that $\rho(\varepsilon)$ is small as $\varepsilon$ goes to 0.

In the same vein, where $M$ is a constant which does not depend on $R$, we have
\begin{equation*}
\left|\sigma_t\right|\leq M\sqrt{\rho(\varepsilon)^2+\left|z_t^2\right|}.
\end{equation*}
Thus, Lemma \ref{eurodance} is applicable and for any $\delta,\rho>0$ and for any $\varepsilon$ small enough, we have
\begin{equation*}
\varepsilon\log\left(\mathbb{P}\left[\sup_{0\leq t\leq\sigma_R}|z_t|\geq\delta\right]\right)\leq B_R+M^2\left(1+\frac{d}{2}\right)+\frac{1}{2}\log\left(\frac{\rho^2}{\rho^2+\delta^2}\right)\,.
\end{equation*}
As $\rho(\varepsilon)$ converges to $0$ as $\varepsilon\searrow0$, we deduce that
\begin{equation*}
\limsup_{\varepsilon\to0}\varepsilon\log\left(\mathbb{P}\left[\sup_{0\leq t\leq\sigma_R}|z_t|\geq\delta\right]\right)=-\infty\,.
\end{equation*}
Now, since
\begin{equation*}
\big\{\left\|X^X_{\varepsilon}-Y^X_{\varepsilon}\right\|_\infty\geq\delta\big\}
\subset
\big\{\sigma_R<1\big\}
\bigcup
\Big\{\sup_{0\leq t\leq\sigma_R}\left|X^X_\varepsilon(t)-Y^X_\varepsilon(t)\right|\geq\delta\Big\}\,,
\end{equation*}
we can conclude as soon as we show that
\begin{equation*}
\lim_{R\to+\infty}\limsup_{\varepsilon\to0}\varepsilon\log\left(\mathbb{P}\left[\sigma_R<1\right]\right)=-\infty\,.
\end{equation*}
By $\tau_R$, we denote the first time that $Y_\varepsilon$ exits from the ball of center $0$ and radius $R$. If, $Y^X_{\varepsilon}(\sigma_R)$ is not in the ball of center $0$ and radius $R+1$, then, we have immediately $\tau_R<1$. Conversely, if $X^X_{\varepsilon}(\sigma_R)$ is not in the ball of center $0$ and radius $R+1$, by taking $\delta<\frac{1}{2}$, we know that with high probability $Y^X_{\varepsilon}(\sigma_R)$ is not in the ball of center $0$ and radius $R$, which means again $\tau_R<1$.

However, $Y^X_{\varepsilon}$ satisfies a Large Deviations Principle for the uniform norm with the good rate function $I^x$. As a consequence,
\begin{align*}
\limsup_{\varepsilon\to0}\varepsilon\log\left(\mathbb{P}\left[\sigma_R<1\right]\right)
&
=\limsup_{\varepsilon\to0}\varepsilon\log\left(\mathbb{P}\left[\tau_R<1\right]\right)
\\
&
\leq-\inf_{\{g\in H^{\otimes d'}\,:\,f(t)=F(g), \left|\left|f\right|\right|_\infty\geq R\}}\frac{1}{2}\int_0^1\left|\dot g\right|^2dt\,.
\end{align*}
It is not difficult to see that the latter expression approaches $-\infty$ as $R$ goes to $\infty$ by using the same arguments as those used from the end of the proof of Lemma~\ref{scooter}.

\end{proof}
Theorem \ref{samsung} now follows easily.
\begin{proof}[Proof of Theorem \ref{samsung}]
Following the methodology described in Section \ref{sec:methodologyForSupSandra} and the above results in combination with \cite[Theorem 2.21]{herrmann2013stochastic} our main Theorem \ref{samsung} follows.
\end{proof}

\subsection[Large Deviations Principle in H\"older Topologies]{Large Deviations Principle in the H\"older Topology for $X_{\varepsilon}$}

\subsubsection{The main result}

Recall the Stochastic process \eqref{sandra1}. We introduce the so-called Skeleton operator $\Phi$ for the MV-SDE \eqref{sandra1} on the Cameron Martin Space $H$, in other words, $\Phi: H^{\otimes d'} \to C([0,1])$ 
\begin{align}
\label{eq:SkeletonOfSandra}
\nonumber
\Phi^x(h)(t) 
= x 
  &+ \int_0^t b(s, \Phi^x(h)(s), \delta_{\Phi^x(0)(s)}) ds\\
	&
  + \int_0^t \sigma(s, \Phi^x(h)(s), \delta_{\Phi^x(0)(s)}) \dot h(s) ds.
\end{align}
The operator $\Phi$ for each $h\in H$ outputs the unique solution to the above ODE. For existence and uniqueness of a solution, see Footnote \ref{ExistenceUniquenessFootnote}. 

Following the same method as in Lemma \ref{BloodyODE} and using the H\"older inequality, one can see that
\begin{align*}
|\Phi^x(h)(t) &- \Phi^x(h)(s)| 
\\
&\leq
\textrm{O}\Big(|t-s|\Big) + M|t-s|^{\tfrac{1}{2}} \sqrt{\int_0^T |\dot{h}(r)|^2 dr}
\leq 
\textrm{O}\Big( |t-s|^{\tfrac{1}{2}}\Big), 
\end{align*}
so $\Phi(h) \in C^{\frac 12}([0,T])$.  We are now able to state the two main results of this section:
\begin{theorem}
\label{theorem:LDPResult}
Let $\alpha\in(0,1/2)$. Let $A$ be a Borel set of the space of $\bR^d$-valued continuous paths over $[0,1]$ in the H\"older topology of $C^\alpha([0,1])$. Let $\Delta(A):=\inf\big\{ {\|\dot h\|_2^2}/{2} ; h\in H^{\otimes d'}, \Phi^x(h)(\cdot)\in A\big\}$. Then
\begin{equation*}
-\Delta(\mathring{A}) 
\leq 
\liminf_{\varepsilon\to0} \varepsilon \log \bP[X_{\varepsilon}^x\in A] 
\leq 
\limsup_{\varepsilon\to0} \varepsilon \log \bP[X_{\varepsilon}^x\in A] 
\leq 
-\Delta(\bar{A}),
\end{equation*}
where $\mathring{A}$ and $\bar{A}$ are the interior and closure of the set $A$ with respect to the topology generated by the H\"older norm. 
\end{theorem}
In order to prove the Theorem \ref{theorem:LDPResult} we first prove another LDP type result (compare with \eqref{eq:Section1Skull}).
\begin{proposition} 
\label{HolderProp}
Let $h\in H^{\otimes d'}$. Take $\forall R, \rho>0$, $\exists \delta, \nu>0$ such that $\forall 0<\varepsilon<\nu$,
$$
\mathbb{P} \Big[ \|X_{\varepsilon}^x - \Phi^x(h)\|_\alpha \geq \rho, \|{\sqrt{\varepsilon}} W-h\|_\infty\leq \delta \Big] 
\lesssim 
\exp\Big( -\frac R \varepsilon \Big).
$$
\end{proposition}
Intuitively, Proposition \ref{HolderProp} (proof given below) quantifies the probability of a highly varying process (in $\|\cdot\|_\alpha$-norm) when the equation's input signal is small (in $\|\cdot\|_\infty$) (see \eqref{eq:Section1Skull}).

\subsubsection{LDP using a Decoupling Argument}
\label{Section:DecouplingArgument}

In this section, we discuss another method for proving the LDP results for $X$ the solution process to  \eqref{sandra1}, called a  decoupling argument. The main idea is to freeze the Law of $X$ in the original MV-SDE and understand the outcome as a standard SDE, with solution $\hat X$, where the coefficients $\hat{b}_\varepsilon(t, x):= b_{\varepsilon}(t, x, \cL_t^{X_{\varepsilon}^x})$ and $\hat{\sigma}_\varepsilon(t, x):= \sigma_{\varepsilon}(t, x, \cL_t^{X_{\varepsilon}^x})$ are just functions of time and space that have no measure dependency. Observe that by Proposition \ref{proposition:ContinuityOfLaw}, the time regularity will not be affected by the measure dependency since we assume $\beta$-H\"older continuity in time for $\beta<\tfrac{1}{2}$, see Assumption \ref{def:He-man+Skeletor}. 

One can prove that $\cL^{X_{\varepsilon}^x} \to \delta_{\psi^x}$ in distribution as ${\varepsilon} \searrow 0$ (where $\psi$ solves \eqref{ode}). The LDP in the uniform topology for the MV-SDE would now follow from a similar LDP under our core conditions for the SDE of $\hat X$. To the best of our knowledge we were unable to find LDP results in H\"older topologies for SDEs with coefficients which allow for time dependency or monotone growth in the spacial variables. Such LDPs do exist, e.g. \cite{arous1994grandes}, present the right LDP but under assumptions of uniform Lipschitzness and uniformly boundedness of $b$ and $\sigma$ plus no time dependency.

Hence the methods and results we present contribute to the MV-SDE literature, but also they are of general interest for the literature on classical SDEs.

\subsubsection{Proofs}
\begin{proof}[Proof of Proposition \ref{HolderProp}]
Let $t\in[0,1]$. Fix $R, \rho>0$. In order to progress with a Local Lipschitz condition, we first need to consider the function $\Phi^x(h)(\cdot)$ (recall \eqref{eq:SkeletonOfSandra}) for $h\in H$. This is a continuous solution of an ODE on the compact interval $[0,1]$. Therefore, it is bounded and we can say that $\exists N>0$ such that $\|\Phi(h)\|_\infty <N$. 

We condition on the event that the process $X_{\varepsilon}^x(\cdot)$ remains in the ball of radius $N$ and we see
\begin{align*}
& \mathbb{P} \Big[ \|X_{\varepsilon}^x - \Phi^x(h)\|_\alpha \geq \rho, \|{\sqrt{\varepsilon}} W-h\|_\infty\leq \delta \Big] 
\\
&\leq
\mathbb{P} \Big[ \|X_{\varepsilon}^x - \Phi^x(h)\|_\alpha \geq \rho, \|{\sqrt{\varepsilon}} W-h\|_\infty\leq \delta , \|X_{\varepsilon}^x\|_\infty<N\Big]
+\bP\Big[ \|X_{\varepsilon}^x\|_\infty\geq N \Big].
\end{align*}
We use that we have the LDP result for $X_{\varepsilon}^x$ in a supremum norm and choose $N$ large enough so that 
$$
\bP \Big[\|X_{\varepsilon}^x\|_\infty \geq N\Big] < \exp\Big( -\frac{R}{\varepsilon}\Big).
$$
We also introduce a step function approximation to discretize the process $X_{\varepsilon}^x$ in \eqref{sandra1} as, for $l\in\bN$ 
$$
X_{\varepsilon}^{x,l}(t) 
= X_{\varepsilon}^x\Big(\frac{j}{l}\Big) \mbox{ on the interval $t\in\Big(\frac{j}{l}, \frac{j+1}{l}\Big]$},
\quad \textrm{ with }X_{\varepsilon}^{x, l}(0) = x.
$$

\emph{Step 1. Analysis of the diffusion term for $h=0$.} Consider
\begin{align}
\label{eq:MV0} 
\mathbb{P}\Big[ \| {\sqrt{\varepsilon}} \int_0^\cdot &\sigma_{\varepsilon}(s, X_{\varepsilon}^x(s), \cL_s^{X_{\varepsilon}^x}) dW(s) \|_\alpha\geq \rho, \|{\sqrt{\varepsilon}} W\|_\infty\leq \delta , \|X_{\varepsilon}^x\|_\infty<N \Big]
\\
\leq &\mathbb{P}\Big[\|{\sqrt{\varepsilon}} \int_0^\cdot [\sigma_{\varepsilon}(s, X_{\varepsilon}^x(s), \cL_s^{X_{\varepsilon}^x})-\sigma_{\varepsilon}(\tfrac{\lfloor sl \rfloor}{l}, X_{\varepsilon}^{x, l}, \cL_{\tfrac{\lfloor sl\rfloor}{l}}^{X_{\varepsilon}^{x}})]dW(s)\|_\alpha \geq \frac{\rho}{2}, \nonumber
\\
& \label{eq:MV1} \hspace{40pt} \tfrac{1}{l^\beta} + \|X_{\varepsilon}^x-X_{\varepsilon}^{x,l}\|_\infty + \bE[ \|X_{\varepsilon}^x-X_{\varepsilon}^{x,l}\|_\infty^2]^{1/2} \leq \gamma \Big]
\\
&+ \label{eq:MV2} \mathbb{P}\Big[\tfrac{1}{l^\beta} +\|X_{\varepsilon}^x-X_{\varepsilon}^{x,l}\|_\infty + \bE[ \|X_{\varepsilon}^x-X_{\varepsilon}^{x,l}\|_\infty^2]^{1/2}> \gamma , \|X_{\varepsilon}^x\|_\infty<N \Big]
\\
&\label{eq:MV3} + \mathbb{P}\Big[\| {\sqrt{\varepsilon}} \int_0^\cdot \sigma_{\varepsilon}(\tfrac{\lfloor sl\rfloor}{l}, X_{\varepsilon}^{x,l}(s), \cL_{\tfrac{\lfloor sl\rfloor}{l}}^{X_{\varepsilon}^{x}})dW(s)\|_\alpha \geq \frac{\rho}{2}, \|{\sqrt{\varepsilon}} W\|_\infty\leq \delta\Big].
\end{align}
We analyze each term in the RHS separately. Firstly, consider the term \eqref{eq:MV1}. We denote 
$$
\eta_\varepsilon = \sup_{s, y, \mu} \Big\{\, |b(s, y, \mu) - b_{\varepsilon}(s, y, \mu)| \, , \, |\sigma(s,y, \mu) -\sigma_{\varepsilon}(s,y, \mu)| \Big\}.
$$
By uniform convergence of $b_{\varepsilon}$ to $b$ and $\sigma_{\varepsilon}$ to $\sigma$, we have that $\lim_{\varepsilon\to0} \eta_\varepsilon = 0$. We choose $\varepsilon$ small enough so that $\eta_\varepsilon \leq \frac{L \gamma}{4}$. Then
\begin{align*}
\eqref{eq:MV1}\leq 
&
\mathbb{P}\Big[ \Big\| \int_0^\cdot \frac{2[\sigma_{\varepsilon}(s, X_{\varepsilon}^x(s), \cL_s^{X_{\varepsilon}^x})-\sigma_{\varepsilon}(\tfrac{\lfloor sl\rfloor}{l}, X_{\varepsilon}^{x,l}(s), \cL_{\frac{\lfloor sl\rfloor}{l}}^{X_{\varepsilon}^{x}})]}{L\gamma} dW(s)\Big\|_\alpha \geq \frac{\rho}{{\sqrt{\varepsilon}} L \gamma}, 
\\
& \hspace{85pt} \frac{2 \|\sigma_{\varepsilon}(\cdot, X_{\varepsilon}^x(\cdot), \cL_{\cdot}^{X_{\varepsilon}^x})-\sigma_{\varepsilon}(\tfrac{\lfloor \cdot l \rfloor}{l}, X_{\varepsilon}^{x,l}(\cdot), \cL_{\cdot}^{X_{\varepsilon}^{x,l}})\|_\infty}{L\gamma} \leq 1\Big]
\\
&\leq C' \exp\Big( \frac{-\rho^2}{C'L^2 \gamma^2 \varepsilon}\Big),
\end{align*}
using Lemma \ref{Lemme2}. Thus choose $\gamma$ such that $\frac{\rho^2}{C'L^2 R} \geq \gamma^2$.

Secondly, consider the term \eqref{eq:MV2}. We take $\varepsilon$ small enough so that $\eta_\varepsilon<1$. Applying It\^o's formula to $|X_{\varepsilon}^x(t)|^2$ gives
\begin{align*}
|X_{\varepsilon}^x(t)|^2 = |x|^2 
&+ 2{\sqrt{\varepsilon}} \int_0^t \langle X_{\varepsilon}^x(s), \sigma_{\varepsilon}(s, X_{\varepsilon}^x(s), \cL_s^{X_{\varepsilon}^x})dW(s) \rangle 
\\
&
+ 2 \int_0^t \langle X_{\varepsilon}^x(s), b_{\varepsilon}(s, X_{\varepsilon}^x(s), \cL_s^{X_{\varepsilon}^x})\rangle ds
\\
&+ \varepsilon \int_0^t \mbox{Tr}\Big( \sigma_{\varepsilon}(s, X_{\varepsilon}^x(s), \cL_s^{X_{\varepsilon}^x}) \sigma_{\varepsilon}(s, X_{\varepsilon}^x(s), \cL_s^{X_{\varepsilon}^x})^{T} \Big) ds.
\end{align*}
Following the estimation methods used to prove Theorem \ref{theo:LocLipMcK-V.ExistUniq} we have 
\begin{align*}
\bE\Big[\, \|X_{\varepsilon}^x\|_\infty^2\Big] &\leq K \Big(|x|^2 +\bE\Big[ ||b(\cdot, 0, \delta_0)||_\infty^2 \Big] \Big) \exp\Big( K+ \bE\Big[ ||b(\cdot, 0, \delta_0)||_\infty^2 \Big]\Big)
<\infty.
\end{align*}
In the same way, we can additionally prove $\bE[\, \|X_{\varepsilon}^x\|_\infty^{2q}] <\infty$.
Let $j=\lfloor t l\rfloor$. We can rewrite our SDE, for $t\in[j/l, (j+1)/l]$, as
$$
X_{\varepsilon}^x(t) - X_{\varepsilon}^x(j/l) = 
{\sqrt{\varepsilon}}\int_{\frac{j}{l}}^t \sigma_{\varepsilon}(s, X_{\varepsilon}^x(s), \cL_s^{X_{\varepsilon}^x})dW(s) 
+ \int_{\frac{j}{l}}^t  b_{\varepsilon}(s, X_{\varepsilon}^x(s), \cL_s^{X_{\varepsilon}^x})ds.
$$
We evaluate the strong error term in the same way as above to see that
\begin{align*}
\sup_{t\in[0,1]} \Big|X_{\varepsilon}^x(t)-X_{\varepsilon}^{x,l}(t)\Big|^2
&
\leq 2 \sup_{t\in\big[\tfrac{j}{l}, \tfrac{j+1}{l}\big]} \Big| \int_{\frac{j}{l}}^t b_{\varepsilon}(s, X_{\varepsilon}^x(s), \cL_s^{X_{\varepsilon}^x})ds \Big|^2   
\\
&\qquad \quad
+ 2\sup_{t\in\big[\tfrac{j}{l}, \tfrac{j+1}{l}\big]} \Big|{\sqrt{\varepsilon}}\int_{\frac{j}{l}}^t \sigma_{\varepsilon}(s, X_{\varepsilon}^x(s), \cL_s^{X_{\varepsilon}^x})dW(s) \Big|^2.
\end{align*}
Taking expectations yields
\begin{align*}
&\bE[\|X_{\varepsilon}^x-X_{\varepsilon}^{x,l}\|_\infty^2] 
\\
&
\leq
2 \bE\Big[ \sup_{t\in\big[\tfrac{j}{l}, \tfrac{j+1}{l}\big]} \Big| \int_{\frac{j}{l}}^t b_{\varepsilon}(s, X_{\varepsilon}^x(s), \cL_s^{X_{\varepsilon}^x})ds \Big|^2\Big]
\\
& \qquad \qquad 
+ 2\bE\Big[\sup_{t\in\big[\tfrac{j}{l}, \tfrac{j+1}{l}\big]} \Big|{\sqrt{\varepsilon}}\int_{\frac{j}{l}}^t \sigma_{\varepsilon}(s, X_{\varepsilon}^x(s), \cL_s^{X_{\varepsilon}^x})dW(s) \Big|^2\Big]
\\
&
\leq \frac{1}{l^2} \Big( 4\eta_\varepsilon + 32L^2 \bE[ ||X_{\varepsilon}^x||_\infty^{2q}] + 8L^2 4\bE[||X_{\varepsilon}^x||_\infty^2] + 4 ||b(\cdot, 0, \delta_0)||_\infty^2 \Big)
+ \frac{8\varepsilon M^2}{l}
\\
&
\lesssim 
\frac{1}{l},
\end{align*}
and we write $\bE[\|X_{\varepsilon}^x-X_{\varepsilon}^{x,l}\|_\infty^2]^{1/2} \leq {K_1}/{\sqrt{l}}$.  In the same way we also have that
\begin{align*}
\|X_{\varepsilon}^x &- X_{\varepsilon}^{x, l}\|_{\infty} 
\\
&= \sup_{j=0,...,l-1}\sup_{t\in[\frac{j}{l}, \frac{j+1}{l}]} \Big| {\sqrt{\varepsilon}}\int_{\frac{j}{l}}^t  \sigma_{\varepsilon}(s, X_{\varepsilon}^x(s), \cL_s^{X_{\varepsilon}^x}) dW(r) \Big|
\\
&
\qquad \qquad 
+ \sup_{j=0,...,l-1}\sup_{t\in[\frac{j}{l}, \frac{j+1}{l}]} \Big| \int_{\frac{j}{l}}^t b_{\varepsilon}(s, X_{\varepsilon}^x(s), \cL_s^{X_{\varepsilon}^x})dr \Big|
\\
&\leq \sup_{j, t} \Big| {\sqrt{\varepsilon}}\int_{\frac{j}{l}}^t  \sigma_{\varepsilon}(s, X_{\varepsilon}^x(s), \cL_s^{X_{\varepsilon}^x})dW(s) \Big| 
\\
&\qquad
+\sup_{j, t} \int_{\frac{j}{l}}^t\Big[ \eta_\varepsilon + L(1+||X||_\infty^q) + L\bE[||X||_\infty^2]^{1/2} +\sup_{r\in[0,1]}
 \Big|b(r, 0, \delta_0)\Big| \Big] ds
\\
&\leq \sup_{j, t} \Big| {\sqrt{\varepsilon}}\int_{\frac{j}{l}}^t  \sigma_{\varepsilon}(s, X_{\varepsilon}^x(s), \cL_s^{X_{\varepsilon}^x})dW(s) \Big| + K_2 \frac{1+ \|X_{\varepsilon}^x\|_\infty^q}{l}.
\end{align*}
Hence we have for term \eqref{eq:MV2} that 
\begin{align*}
\mathbb{P}\Big[&\tfrac{1}{l^\beta} + \|X_{\varepsilon}^x-X_{\varepsilon}^{x,l}\|_\infty + \bE[ \|X_{\varepsilon}^x-X_{\varepsilon}^{x,l}\|_\infty^2]^{1/2}> \gamma , \|X_{\varepsilon}^x\|_\infty<N \Big]
\\
&\leq \mathbb{P}\Big[ \sup_{j, t} \Big| {\sqrt{\varepsilon}}\int_{\frac{j}{l}}^t  \sigma_{\varepsilon}(s, X_{\varepsilon}^x(s), \cL_s^{X_{\varepsilon}^x})dW(s) \Big| 
\\ & \hspace{2cm}  + K_2 \frac{1+ \|X_{\varepsilon}^x\|_\infty^q}{l} +\frac{K_1}{\sqrt{l}} + \frac{1}{l^\beta} >\gamma, 
\|X_{\varepsilon}^x\|_\infty<N \Big]
\\
&\leq \mathbb{P}\Big[ \sup_{j, t} \Big| {\sqrt{\varepsilon}}\int_{\frac{j}{l}}^t  \sigma_{\varepsilon}(s, X_{\varepsilon}^x(s), \cL_s^{X_{\varepsilon}^x})dW(s) \Big| > \gamma - \frac{K_3}{l^{\frac12 \wedge \beta}}\Big],
\end{align*}
where $K_3 = K_1+K_2(1+N^q)+1$. 

Therefore, using Chernoff's inequality
\begin{align*}
\eqref{eq:MV2}
&
\leq\mathbb{P}\Big[ \sup_{j, t} 
 \exp\Big(\lambda \sup_{j, t} \Big| {\sqrt{\varepsilon}}\int_{\frac{j}{l}}^t  \sigma_{\varepsilon}(s, X_{\varepsilon}^x(s), \cL_s^{X_{\varepsilon}^x})dW(s) \Big| \Big) 
\\
& \hspace{6cm}         > \exp\Big(\frac{\lambda}{l^{\frac12 \wedge \beta}} \big(\gamma l^{\frac12 \wedge \beta} - K_3\big)\Big)\Big]
\\
&
\leq \frac{\sup_{j, t}
\mathbb{E}\Big[  \exp\Big(\lambda \Big| {\sqrt{\varepsilon}}\int_{\frac{j}{l}}^t  \sigma_{\varepsilon}(s, X_{\varepsilon}^x(s), \cL_s^{X_{\varepsilon}^x})dW(s) \Big| \Big) \1_{\|X_{\varepsilon}^x\|_\infty<N} \Big]}
	{\exp\Big(\frac{\lambda}{l^{\frac12 \wedge \beta}} (\gamma l^{\frac12 \wedge \beta}- K_3)\Big)} 
\\
&
\lesssim \exp\Big(\lambda^2 \varepsilon \frac{M}{2l} - \frac{\lambda}{l^{\frac12 \wedge \beta}} (\gamma l^{\frac12 \wedge \beta}- K_3)\Big)
\\
&
\lesssim \exp\Big( \frac{-(\gamma l^{\frac12 \wedge \beta} - K_3)^2}{2 \varepsilon M} \frac{l}{l^{\frac12 \wedge \beta}} \Big),
\end{align*}
by optimizing over the arbitrary choice of $\lambda$. We can now choose the constant $l$ such that $\frac{(\gamma l^{\frac12 \wedge \beta} -K_3)^2}{2M}l^{1- (1 \wedge 2\beta)}>R$. 

Finally, to evaluate Equation \eqref{eq:MV3}, we first consider $\sigma_{\varepsilon}(X_{\varepsilon}^{x, l}(\cdot), \cL^{X_{\varepsilon}^{x}}_{\tfrac{\lfloor \cdot l\rfloor }{l}})$. This process is constant over the interval $(\frac{j}{l}, \frac{j+1}{l}]$. Then taking the H\"older norm we get 
\begin{align*}
&\Big\|\int_0^\cdot \sigma_{\varepsilon}(\tfrac{\lfloor ls \rfloor}{l}, X_{\varepsilon}^{x, l}(s), \cL^{X_{\varepsilon}^{x}}_{\tfrac{\lfloor sl\rfloor }{l}}) dW(s)\Big\|_\alpha 
\\
&\quad = \Big\| \sum_{j=0}^{l-1} \sigma_{\varepsilon}\big(\tfrac{j}{l}, X_{\varepsilon}^{x, l}(\tfrac{j}{l}), \cL^{X_{\varepsilon}^{x}}_{\tfrac{j}{l}}\big) \big[W(\frac{j+1}{l}\wedge \cdot) - W(\frac{j}{l}\wedge \cdot)\big] \Big\|_\alpha \leq 2lM \|W\|_\alpha,
\end{align*}
using $\|\sigma_{\varepsilon} \|_\infty\leq M$ and the triangle inequality. Then
\begin{align*}
\mathbb{P}\Big[
\|{\sqrt{\varepsilon}} \int_0^\cdot& \sigma_{\varepsilon}(s, X_{\varepsilon}^{x, l}(s), \cL_s^{X_{\varepsilon}^{x, l}}) dW(s)\|_{\alpha} \geq \frac{\rho}{2}, \|{\sqrt{\varepsilon}} W\|_\infty\leq \delta\Big]
\\
&
\leq \mathbb{P}\Big[ \|W\|_\alpha\geq \frac{\rho}{4{\sqrt{\varepsilon}} lM}, 
                         \|W\|_\infty \leq \frac{\delta}{{\sqrt{\varepsilon}}}\Big]
\\
&
\leq C \max \Big( 1, \Big(\frac{\rho}{4Ml\delta}\Big)^{1/\alpha} \Big) \exp\Big( \frac{-1}{\varepsilon} \frac{1}{C} \Big(\frac{\rho}{4Ml\delta}\Big)^{1/\alpha} \delta^2\Big),
\end{align*}
where we applied Lemma \ref{Lemme1} and chose $\delta$ such that $\frac{\rho}{R^{\alpha}4Ml C^{\alpha}}\geq \delta^{1-2\alpha}$. 

Injecting these three results in \eqref{eq:MV0} gives us that
\begin{align}
\nonumber
\label{eq:MV4} 
\mathbb{P}\Big[ \| {\sqrt{\varepsilon}} \int_0^\cdot \sigma_{\varepsilon} (s, X_{\varepsilon}^x(s), \cL^{X_{\varepsilon}^{x}}_{\tfrac{\lfloor sl\rfloor }{l}}) dW(s) \|_\alpha\geq \rho, 
&\|{\sqrt{\varepsilon}} W\|_\infty
 \leq \delta, \|X_{\varepsilon}^x\|_\infty<N \Big]
\\
& \lesssim \exp\Big( -\frac{R}{\varepsilon}\Big).
\end{align}

\emph{Step 2. The H\"older norm of the whole process when $h=0$.} We have 
\begin{align}
\|X_{\varepsilon}^x - \Phi^x(0)\|_{\alpha, t}  
\label{eq:MV5}
& \leq\big\|{\sqrt{\varepsilon}} \int_0^\cdot \sigma_{\varepsilon}(s, X_{\varepsilon}^x(s), \cL_s^{X_{\varepsilon}^x}) dW(s)\big\|_{\alpha, t} 
\\
& \label{eq:MV6}+ \big\|\int_0^\cdot [b_{\varepsilon}(s, X_{\varepsilon}^x(s), \cL_s^{X_{\varepsilon}^x}) - b(s, X_{\varepsilon}^x(s), \cL_s^{X_{\varepsilon}^x})]ds \big\|_{\alpha, t} 
\\
& \label{eq:MV7}+ \big\|\int_0^\cdot [b(s, X_{\varepsilon}^x(s), \cL_s^{X_{\varepsilon}^x}) - b(s, \Phi^x(0)(s), \delta_{\Phi^x(0)(s) })]ds\big\|_{\alpha, t}. 
\end{align}
Equation \eqref{eq:MV5} is the term in \eqref{eq:MV4} that we desire. Equation \eqref{eq:MV6} is bounded above by $\eta_\varepsilon$. We only consider the cases when $\|X_{\varepsilon}^x\|_\infty, \|\Phi^x(0)\|_\infty<N$ since we know that $\Phi^x(0)(t)$ remains in this ball and we conditioned on $X_{\varepsilon}^x(t)$ remaining in the same ball. This means that by the Locally Lipschitz condition, we can say that $b(t, x, \mu)$ is Lipschitz in the spacial variable with constant $L_N$. Therefore for \eqref{eq:MV7} we have
\begin{align}
\Big\|\int_0^\cdot b(s, X_{\varepsilon}^x(s)&, \cL_s^{X_{\varepsilon}^x})  - b(s, \Phi^x(0)(s), \delta_{\Phi^x(0)(s)})ds\Big\|_{\alpha, t} \nonumber
\\
\leq& \sup_{p,q\in[0,t]} \frac{\int_p^q |b(s, X_{\varepsilon}^x(s), \cL_s^{X_{\varepsilon}^x})-b(s, \Phi^x(0)(s), \delta_{\Phi^x(0)(s)})|}{|q-p|^\alpha} \nonumber
\\
\leq&\sup_{p,q\in[0,t]} \frac{L_N}{|q-p|^\alpha} \int_p^q |X_{\varepsilon}^x(s)-\Phi^x(0)(s)|ds \nonumber
\\
& \qquad \qquad + \frac{L}{|q-p|^\alpha} \int_p^q \bE[|X_{\varepsilon}^x(s)-\Phi^x(0)(s)|^2]^{1/2}ds \nonumber
\\
\label{eq:MV8}
\leq& L_N\|X_{\varepsilon}^x(\cdot) - \Phi^x(0)(\cdot)\|_{\infty, t} + L_N \int_0^{t} \|X_{\varepsilon}^x(\cdot)-\Phi^x(0)(\cdot)\|_{\alpha, s}ds 
\\
\label{eq:MV9}
&\qquad \qquad +L \bE[\|X_{\varepsilon}^x-\Phi^x(0)\|_\infty^2]^{1/2},
\end{align}
using Lemma \ref{lemma:WassersteinAgainstDirac}. 

Next, we want to show that the strong error $\bE[\|X_{\varepsilon}^x-\Phi(0)\|_\infty^2]$ can be controlled by $\varepsilon$. Using that
\begin{align*}
d\big(X_{\varepsilon}^x-\Phi^x(0)\big)(t) =& \sigma_{\varepsilon}\Big(t, X_{\varepsilon}^x(t), \cL_t^{X_{\varepsilon}^x}\Big)dW(t) 
\\
&+ \Big(b_{\varepsilon}(t, X_{\varepsilon}^x(t), \cL_t^{X_{\varepsilon}^x}) - b(t, X_{\varepsilon}^x(t),\cL_t^{X_{\varepsilon}^x})\Big)dt 
\\
&+ \Big(b(t, X_{\varepsilon}^x(t), \cL_t^{X_{\varepsilon}^x}) - b(t, \Phi^x(0)(t), \delta_{\Phi^x(0)(t)})\Big)dt,
\end{align*}
and It\^o's formula for $f(x) = |x|^2$ with $X_{\varepsilon}^x(0)-\Phi^x(0)(0)=0$ gives that
\begin{align*}
||X_{\varepsilon}^x &- \Phi^x(0) ||_{\infty, t}^2
\\
& \leq 2{\sqrt{\varepsilon}} \sup_{0\leq s\leq t} \Big| \int_0^s  \langle X_{\varepsilon}^x(r) - \Phi^x(0)(r), \sigma_{\varepsilon}(r, X_{\varepsilon}^x(r), \cL_r^{X_{\varepsilon}^x})dW(r) \rangle \Big| 
\\
&+ \varepsilon M^2 td  + \int_0^t 2\eta_\varepsilon |X_{\varepsilon}^x(r) - \Phi^x(0)(r)| dr 
\\
&+ 2\int_0^t \Big| \langle X_{\varepsilon}^x(r) - \Phi^x(0)(r), b(r, X_{\varepsilon}^x(r), \cL_r^{X_{\varepsilon}^x}) - b(r, \Phi^x(0)(r), \delta_{\Phi^x(0)(r)}) \rangle \Big| dr.
\end{align*}
Squaring and taking expectations gives
\begin{align*}
 & \bE\Big[ \|X_{\varepsilon}^x  - \Phi^x(0) \|_{\infty, t}^4\Big] 
\\
& \leq
 64pM^2 \varepsilon \int_0^t \Big( \bE[ \|X_{\varepsilon}^x - \Phi^x(0)\|_{\infty, r}^4 ]+1 \Big) dr + 4\varepsilon^2 M^4 t^2
\\
&\quad + 16t\eta_\varepsilon \int_0^t \Big( \bE[ \|X_{\varepsilon}^x - \Phi^x(0)\|_{\infty, r}^4 ]+1\Big) dr
+ 16t \int_0^t 4L^2 \bE[ \|X_{\varepsilon}^x - \Phi^x(0)\|_{\infty, r}^4] dr.
\end{align*}
Refining, we then obtain
$\bE\big[ \|X_{\varepsilon}^x-\Phi^x(0)\|_\infty^2\big]^{1/2} 
\leq 
K (\eta_\varepsilon^{1/4} \vee \varepsilon^{1/4} ) e^{K}.
$ 
We have shown that this expectation is of order $\varepsilon^{1/4}$. Now we consider $\|X_{\varepsilon}^x-\Phi^x(0)\|_\infty$. Since the supremum norm can be made to appear inside the integrals, we have
\begin{align*}
\|X_{\varepsilon}^x-\Phi^x(0)\|_{\infty, t} \leq&
\|{\sqrt{\varepsilon}} \int_0^\cdot \sigma_{\varepsilon}(s, X_{\varepsilon}^x(s), \cL_s^{X_{\varepsilon}^x}) dW(s)\|_{\infty, t} + \eta_\varepsilon t 
\\
&+ \int_0^t L_N \| X_{\varepsilon}^x - \Phi(0) \|_{\infty, r} dr + Lt \bE[ \|X_{\varepsilon}^x-\Phi^x(0)\|_{\infty}^2]^{1/2}, 
\end{align*}
and by using Gr\"onwall, we get
\begin{align}
\|X_{\varepsilon}^x-\Phi^x(0)\|_{\infty, t} \leq
&\Big( \|{\sqrt{\varepsilon}} \int_0^\cdot \sigma_{\varepsilon}(s, X_{\varepsilon}^x(s), \cL_s^{X_{\varepsilon}^x}) dW(s)\|_{\infty, t} \nonumber
\\
&\qquad \qquad + \big( \eta_\varepsilon + \bE[\|X_{\varepsilon}^x - \Phi^x(0)\|_\infty ^2]^{1/2}\big)\Big)e^{L_N t} \nonumber
\\
\label{eq:MV11}
\leq&\Big( \|{\sqrt{\varepsilon}} \int_0^\cdot \sigma_{\varepsilon}(s, X_{\varepsilon}^x(s), \cL_s^{X_{\varepsilon}^x}) dW(s)\|_{\alpha, t} + (\varepsilon^{1/4} \vee \eta_\varepsilon^{1/4})K' )\Big)e^{K't}.
\end{align}

Combining Equation \eqref{eq:MV8}, Equation \eqref{eq:MV9} and Equation \eqref{eq:MV11} gives
\begin{align*}
\|X_{\varepsilon}^x-\Phi^x(0)\|_\alpha 
\leq&  \| {\sqrt{\varepsilon}}\int_0^\cdot \sigma_{\varepsilon}(s, X_{\varepsilon}^x(s), \cL_s^{X_{\varepsilon}^x})dW(s)\|_\alpha \Big( 1+ L_N e^{K'}\Big) e^{L_N} 
\\
&+(\eta_\varepsilon^{1/4} \vee \varepsilon^{1/4} )\Big( 1+ L_NK'e^{K'} + KLe^K \Big)e^{L_N}
\\
\leq& \Big[\| {\sqrt{\varepsilon}}\int_0^\cdot \sigma_{\varepsilon}(s, X_{\varepsilon}^x(s), \cL_s^{X_{\varepsilon}^x})dW(s)\|_\alpha +(\eta_\varepsilon^{1/4} \vee \varepsilon^{1/4}) \Big] K_4.
\end{align*}
Thus for any choice of $\rho$ we see that
\begin{align*}
\mathbb{P} &\Big[ \|X_{\varepsilon}^x - \Phi^x(0)\|_\alpha \geq \rho , \|X_{\varepsilon}^x\|_\infty<N \Big]
\\
&\qquad 
\leq
\mathbb{P} \Big[ \Big(\| {\sqrt{\varepsilon}}\int_0^\cdot \sigma_{\varepsilon}(s, X_{\varepsilon}^x(s), \cL_s^{X_{\varepsilon}^x})dW(s)\|_\alpha +(\eta_\varepsilon^{1/4} \vee \varepsilon^{1/4}) \Big) K_4\geq \rho , \|X_{\varepsilon}^x\|_\infty<N \Big].
\end{align*}
and by choosing $\varepsilon$ small enough such that $(\eta_\varepsilon^{1/4} \vee \varepsilon^{1/4})<\frac{\rho}{2K_4}$ we get
\begin{align*}
\mathbb{P} \Big[& \|X_{\varepsilon}^x - \Phi^x(0)\|_\alpha \geq \rho, \|{\sqrt{\varepsilon}} W\|_\infty \leq \delta, \|X_{\varepsilon}^x\|_\infty<N \Big] 
\\
&\leq 
\mathbb{P} \Big[ \| {\sqrt{\varepsilon}}\int_0^\cdot \sigma_{\varepsilon}(s, X_{\varepsilon}^x(s), \cL_s^{X_{\varepsilon}^x})dW(s)\|_\alpha \geq \frac{\rho}{2K_4}, \|{\sqrt{\varepsilon}} W\|_\infty\leq \delta, \|X_{\varepsilon}^x\|_\infty<N \Big]
\\
&
\lesssim
\exp\Big(-\frac{R}{\varepsilon}\Big),
\end{align*}
since in Equation \eqref{eq:MV4} the choice of $\rho$ is arbitrary. 

\emph{Step 3. The case when $h\neq 0$.} For the final step, we use the same method as in \cite{arous1994grandes} to extend the results to Wiener processes with drift. Using a Girsanov transformation we have that there is a measure $\tilde{\bP}$ absolutely continuous to the standard probability measure $\bP$. 

Note that the law of the stochastic process is not changed by perturbing the path of the Brownian motion by some element of the Cameron Martin space. When solving a McKean-Vlasov equation (unlike classical SDEs), one has to fix the law of the probability space in order to define $\cL^X = \bP \circ X^{-1}$. Hence the law is not changed when one considers a different driving noise for the SDE. This is most obvious in expression \eqref{eq:SkeletonOfSandra} where the delta distribution follows the path of the Skeleton with input $h=0$. 

We rewrite the SDE and Skeleton process 
\begin{align*}
X_{\varepsilon}^x(t) =& x + \int_0^t b_{\varepsilon} (s, X_{\varepsilon}^x(s), \bP \circ [X_{\varepsilon}^x(s)]^{-1}) ds
+ \int_0^t \sigma_{\varepsilon}(s, X_{\varepsilon}^x(s), \bP \circ [X_{\varepsilon}^x(s)]^{-1}) \dot{h}(s) ds
\\
&+ {\sqrt{\varepsilon}} \int_0^t \sigma_{\varepsilon}(s, X_{\varepsilon}^x(s), \bP \circ [X_{\varepsilon}^x(s)]^{-1}) d\tilde{W}(s),
\\
\Phi^x(h)(t) =& x + \int_0^t b(s, \Phi^x(h)(s), \delta_{\Phi^x(0)(s)} ) ds + \int_0^t \sigma(s, \Phi^x(h)(s), \delta_{\Phi^x(0)(s)}) \dot{h}(s) ds,
\end{align*}
where $\tilde{W} = W - {h}/{{\sqrt{\varepsilon}}}$, $\tilde{\bP}$ is the measure where $\tilde{W}$ is a Brownian motion and $\bP \circ [X_{\varepsilon}^x(t)]^{-1} = \cL_t^{X_{\varepsilon}^x}$ . The drift term $b_{\varepsilon} + \sigma_{\varepsilon}\dot{h}$ satisfies the properties from before and matches the Skeleton process $\Phi^x(h)$. 

Also note that
$$
W^{(2)} \Big( \bP \circ [X_{\varepsilon}^x(t)]^{-1}, \delta_{\Phi^x(0)(t)} \Big) = \bE^{\bP}\Big[ |X_{\varepsilon}^x(t) - \Phi(0)(t)|^2 \Big]^{\tfrac{1}{2}},
$$
which we have already showed to go to 0 as $\varepsilon\to 0$. Thus we argue in the same way as in Step 2 and conclude
\begin{align*}
\mathbb{P} \Big[ \|X_{\varepsilon}^x &- \Phi^x(h)\|_\alpha \geq \rho, \|{\sqrt{\varepsilon}} W - h\|_\infty \leq \delta \Big] \\
&\lesssim \tilde{\mathbb{P}} \Big[ \|\tilde{X_{\varepsilon}^x} - \Phi^x(0)\|_\alpha \geq \rho, \|{\sqrt{\varepsilon}} \tilde{W} \|_\infty \leq \delta \Big] 
\lesssim \exp\Big( -\frac{R}{\varepsilon}\Big).
\end{align*} 

\end{proof}

We are now in position to prove our 2nd main result, Theorem \ref{theorem:LDPResult}.
\begin{proof}[Proof of Theorem \ref{theorem:LDPResult}]
\emph{Proving the upper bound.} 
First consider the case where $0\notin A$ and $A$ is closed in the H\"older Topology. Then there exists an $r$ such that $\Delta(A)>r>0$. Let us consider the ball in the Cameron-Martin space $H$
$$
\Big\{ h\in H^{\otimes d'} : h(t)=\int_0^t \dot h(s)ds, \frac{\|\dot h\|_2^2}{2}\leq r\Big\}.
$$
Recall that if $h\in H^{\otimes d'}$ then $h\in C^{\frac12}([0,1]; \bR^m)$ and is bounded and, moreover, that $\|h\|_\infty \leq \|h\|_{\frac12} \leq \|\dot h\|_{2}$. Therefore we can apply Arzel\`a-Ascoli Theorem \cite{Dunford1988Linear} to get that this set is compact. Hence we can find a finite open cover of this set and we can restrict the radius of the open balls. We write
$$
\Big\{ h\in H^{\otimes d'} : h(t)=\int_0^t \dot h(s)ds, \frac{\|\dot h\|_2^2}{2}\leq r\Big\} 
\subset 
\bigcup_{i=1}^N B_\infty(h_i, \delta_{h_i})=U.
$$
These balls are in the uniform topology and the elements $h_i$ are all have ${\|\dot h\|_2^2}/{2}<r$. By this property, $\Phi(h_i)\notin A$. If it were, ${\|\dot h\|_2^2}/{2}>\Delta(A)$. The set $A$ is closed in the H\"older topology so $A^c$ is open in the H\"older Topology. Therefore, there exists a $\rho_{h_i}$ such that in the H\"older Topology $B_\alpha(\Phi(h_i), \rho_{h_i})$ is in $A^c$, and therefore does not intersect with $A$. Hence when $X_{\varepsilon}^x\in A$, we can say that $\|X_{\varepsilon}^x - \Phi^x(h_i)\|_\alpha \geq \rho_{h_i}$. Finally, we can estimate  
\begin{align*}
\bP[ X_{\varepsilon}^x \in A]
&
= \bP[X_{\varepsilon}^x \in A, {\sqrt{\varepsilon}} W \notin U] + \bP[X_{\varepsilon}^x \in A, {\sqrt{\varepsilon}} W \in U]
\\
&
\leq \bP[{\sqrt{\varepsilon}} W \notin U] + \bP[X_{\varepsilon}^x \in A, {\sqrt{\varepsilon}} W \in U]
\\
&
\leq \bP[{\sqrt{\varepsilon}} W \notin U] + \sum_{i=1}^N \bP[\|X_{\varepsilon}^x-\Phi^x(h_i)\|_\alpha\geq \rho_{h_i}, \|{\sqrt{\varepsilon}} W- h_i\|_\infty 
\leq \delta_{h_i}]
\\
&
\leq\bP[{\sqrt{\varepsilon}} W \notin U] + N \exp\Big(-\frac{2r}{\varepsilon}\Big),
\end{align*}
where for the last line we apply Proposition \ref{HolderProp}, with $\delta_{h_i}$ and $\varepsilon$ chosen sufficiently small for the given $\rho_{h_i}$. The $\delta_{h_i}$ are dependent on our choice of open cover for the compact set, so we can make them as small as required. We already have a Large Deviation Principle for a Wiener process on the uniform norm by \cite{herrmann2013stochastic}. Hence we have for $\varepsilon$ sufficiently small that
$$
\bP[ {\sqrt{\varepsilon}} W\notin U] \leq \exp\Big( -\frac{\Delta (U^c)}{\varepsilon}\Big).
$$
If $h\notin U$, then we have that $\frac{\|\dot h\|_2^2}{2}>r$ and consequently
$\bP[ {\sqrt{\varepsilon}} W\notin U] \leq \exp\Big( -{2r}/{\varepsilon}\Big)$. 
Combining all of this together we get
$$
\limsup_{\varepsilon\to0} \varepsilon \log(\bP[X_{\varepsilon}^x\in A]) \leq -r,
$$
where $r$ was chosen arbitrarily so that $r<\Delta(A)$ where $A$ is closed. We optimize for our choice of $r$ and get
$$
\limsup_{\varepsilon\to0} \varepsilon \log(\bP[X_{\varepsilon}^x\in A]) \leq -\Delta(A),
$$
which is the upper inequality for the Theorem. 

\emph{Proving the lower bound.}  Now consider $A$ to be an open set in the H\"older topology and let $h\in H^{\otimes d'}$ such that $\Phi^x(h)\in A$. There exists a $\rho>0$ such that the H\"older ball $B_\alpha(\Phi^x(h), \rho) \subset A$. Also we have that
\begin{align*}
\bP[\|{\sqrt{\varepsilon}} W &-h\|_\infty <\delta]
\\
& \leq \bP[\|X_{\varepsilon}^x-\Phi^x(h)\|_\alpha \geq \rho, \|{\sqrt{\varepsilon}} W -h\|_\infty \leq \delta] + \bP[\|X_{\varepsilon}^x-\Phi^x(h)\|_\alpha < \rho].
\end{align*}
Hence
\begin{align*}
\bP[X_{\varepsilon}^x\in A] \geq& \bP[\|X_{\varepsilon}^x-\Phi^x(h)\|_\alpha < \rho] 
\\
\geq& \bP[\|{\sqrt{\varepsilon}} W -h\|_\infty <\delta] - \bP[\|X_{\varepsilon}^x-\Phi^x(h)\|_\alpha \geq \rho, \|{\sqrt{\varepsilon}} W -h\|_\infty \leq \delta]
\\
\geq&\bP[\|{\sqrt{\varepsilon}} W -h\|_\infty <\delta] - \exp\Big(-\frac{R}{\varepsilon}\Big).
\end{align*}
Applying the LDP for the Brownian motion (see Lemma \ref{camion}) and using that $\frac{\|\dot h\|_2^2}{2}\geq \Delta(B_\infty(h, \delta))$, we see that
$$
\bP[\|{\sqrt{\varepsilon}} W-h\|_\infty < \delta] \geq \exp\Big(-\frac{\|\dot h\|_2^2}{2\varepsilon}\Big)
$$
and hence 
$$
\bP[X_{\varepsilon}^x\in A]
\geq \exp\Big( -\frac{\|\dot h\|_2^2}{2\varepsilon}\Big) - \exp\Big(-\frac{R}{\varepsilon}\Big),
$$
where we can choose $R$ to take any value. Choosing $R=\|\dot h\|_2^2$ and rearranging we get
$$
\bP[X_{\varepsilon}^x\in A]
\geq 
\exp\Big( -\frac{\|\dot h\|_2^2}{2\varepsilon}\Big) \Big( 1-\exp\Big( -\frac{\|\dot h\|_2^2}{2\varepsilon}\Big)\Big). 
$$
Hence
$$
\liminf_{\varepsilon\to 0} \varepsilon \log(\bP[X_{\varepsilon}^x\in A] ) 
\geq \liminf_{\varepsilon\to 0} \varepsilon \log\Big(1-\exp\Big( -\frac{\|\dot h\|_2^2}{2\varepsilon}\Big)\Big) -\frac{\|\dot h\|_2^2}{2}.
$$
The limit goes to $0$ for any choice of $h\in H^{\otimes d'}$. Finally, as $h$ was arbitrarily chosen in $A$, we take the infimum over $h$ and get
$$
\liminf_{\varepsilon\to 0} \varepsilon \log(\bP[X_{\varepsilon}^x\in A] ) \geq - \Delta(A).
$$
This completes the proof of the theorem. 
\end{proof}

\section{Functional Iterated Logarithm Law}
\label{sec:StrassenSection}

Strassen's Law, or the Law of Iterated Logarithm describes the magnitude of the fluctuations of a Brownian motion. It was first proved in \cite{strassen1964invariance}. Observe that for a Brownian Motion $W(t)$, we have that $X_n^{(1)} (t) = {W(nt)}/{n} \to 0$ as $n\to \infty$ both in probability and almost surely. However $X_n^{(2)} (t) = {W(nt)}/{\sqrt{n}}$ is also a Brownian Motion for any choice of $n$. Therefore, something is happening between $n$ and $\sqrt{n}$ which is turning a stochastic process into a deterministic constant in the limit as $n\to \infty$. Strassen's Law says that
$$
X_n^{(3)}(t) = \frac{W(nt)}{\sqrt{n\log\log(n)}},
$$
converges to 0 in probability but does not converge almost surely. In particular
$$
\limsup_{n\to \infty} X_n^{(3)} (1) = \sqrt{2},\quad \textrm{almost surely.}
$$
In this section we are interested in studying whether stochastic processes have a similar type of property. We will consider the solution of the SDE run over a large time interval of order $n$ and rescaled to order $\sqrt{n \log(\log(n))}$. Similar to the proof of Strassen's Law, we will show that the set of rescaled paths is relatively compact in the H\"older topology but that the set of limit points of this set is uncountable which implies the failure of almost sure convergence. 

In \cite{baldi1986large}, Baldi proves a Law of Iterated Logarithm for classical SDEs for the uniform topology. This was then extended in \cite{eddahbi2000large} and later \cite{eddahbi2002strassen} to other coarser pathspace topologies. Standard LDP results easily give us convergence in probability. We calculate the set of possible limit points of the scaled diffusions which for a classical SDE are 
\begin{align*}
\Big\{ \Phi^x(h) : d\Phi^x(h)(t) &= b\big(\Phi^x(h)(t)\big) dt + \sigma\big(\Phi^x(h)(t)\big) \dot{h}(t)dt,\\
&  \Phi^x(h)(0) = x \mbox{ and } \|\dot{h}\|_2\leq \sqrt{2}\Big\}.
\end{align*}
We show below, that similarly for a McKean-Vlasov SDE these are
\begin{align*}
\Big\{ \Phi^x(h) : 
d\Phi^x(h)(t) & = b\big(\Phi^x(h)(t), \delta_{\Phi^x(0)(t)}\big) dt 
                + \sigma\big(\Phi^x(h)(t), \delta_{\Phi^x(0)(t)}\big) \dot{h}(t)dt, 
\\
&
\Phi^x(h)(0) 
=
 x \mbox{ and } ||\dot{h}||_2\leq \sqrt{2}\Big\}.
\end{align*}
We will follow the methods of \cite{baldi1986large}, \cite{eddahbi2000large} and \cite{eddahbi2002strassen} to extend the LDP results to prove an Iterated Logarithm Law for the class of McKean-Vlasov SDEs in Theorem \ref{theo:LocLipMcK-V.ExistUniq}. It seems possible to use microscopic rescaling of the Brownian motion such as in \cite{gantert1993inversion} to provide an alternative proof of our result, however, we do not pursue this point.

\begin{remark}[Decoupling Argument]
\label{rem:NoDecoupling}
In this section, we are unable to use a decoupling argument as highlighted in Section \ref{Section:DecouplingArgument}.  

To the best of our knowledge, there are no results proving a Strassen type law for SDEs with coefficients which can vary in time and we were unable to establish any such results while working on this paper. The conditions that we require on the measure dependency are similar to those of the spacial dependency and do not naturally translate into conditions for time dependency. Therefore, proving that they are satisfied is much easier in the MV-SDE setting when they are written as properties on the measure dependency than for some general time dependent coefficient.

\end{remark}

\subsection*{Functional Iterated Logarithm Law for McKean-Vlasov SDEs}

Firstly, we need to define in what sense we will be rescaling our MV-SDE. 
\begin{definition}
\label{dfn:SystemContraction}
Let $\alpha\in \bR^+$. A family of continuous bijections $\Gamma_\alpha : \bR^d \to \bR^d$ is said to be a System of Contractions centered at $x$ if
\begin{enumerate}
\item $\Gamma_\alpha(x) = x$ for every $\alpha \in \bR^+$. 
\item If $\alpha\geq \beta$ then $| \Gamma_\alpha(y_1) - \Gamma_\alpha(y_2) - \Gamma_\alpha(z_1) +\Gamma_\alpha(z_2)| \leq |\Gamma_\beta(y_1) - \Gamma_\beta(y_2) - \Gamma_\beta(z_1) +\Gamma_\beta(z_2)|$ for every $y_1, y_2, z_1, z_2 \in \bR^d$. 
\item $\Gamma_1$ is the identity and $(\Gamma_\alpha)^{-1} = \Gamma_{\alpha^{-1}}$. 
\item For every compact set $\cK\subset C^\alpha([0,1]; \bR^d)$, $f\in \cK$ and $\varepsilon>0$, $\exists \delta>0$ such that $|pq - 1|<\delta$ implies
$$
\|\Gamma_p \circ \Gamma_q (f) - f\|_\alpha < {\sqrt{\varepsilon}},\qquad \forall\ p,q\in\bR^+.
$$
\end{enumerate}
\end{definition}
The simplest example of such a system of contractions is $\Gamma_\alpha(y) = \tfrac{y}{\alpha}$ centered at $x=0$. Indeed this is the specific operator used when proving Strassen's Law for Brownian motion. Also note that we only really care about $\Gamma_\alpha$ for $\alpha>1$. It is clear that for $\alpha<1$, the operators $\Gamma_\alpha$ will not be contraction operators. 

\begin{example}
In fact, a linear contraction operator with a transformation will satisfy these conditions. Consider for example 
$\Gamma_\alpha(y) = \tfrac{(y-x)}{\alpha} + x$ and naturally, $\Gamma_\alpha(x) = x$. Similarly, for $\alpha \geq \beta$
\begin{align*}
\Gamma_\alpha(y_1) & - \Gamma_\alpha(y_2) - \Gamma_\alpha(z_1) +\Gamma_\alpha(z_2) 
\\
&
= \frac{y_1-x}{\alpha} + x - \frac{y_2-x}{\alpha} - x - \frac{z_1-x}{\alpha} - x + \frac{z_2-x}{\alpha} + x
\\
&\leq \frac{y_1-y_2-z_1+z_2}{\beta}
=\Gamma_\beta(y_1) - \Gamma_\beta(y_2) - \Gamma_\beta(z_1) +\Gamma_\beta(z_2)
\end{align*}
Finally, for $|pq - 1|<\delta$ we have 
\begin{align*}
\|\Gamma_p \circ \Gamma_q (f) - f\|_\alpha  
&
= \sup_{s, t\in[0,1]} \frac{|\Gamma_p \circ \Gamma_q (f(t)) - f(t) - \Gamma_p \circ \Gamma_q (f(s)) + f(s)|}{|t-s|^\alpha}
\\
&
=\sup_{s, t\in[0,1]} \frac{\Big| \big[ \frac{f(t)}{p q} - f(t) \big] - \big[\frac{f(s)}{p q} - f(s) \big]\Big|}{|t-s|^\alpha}
\\
&
= \Big|\frac{1}{pq} - 1\Big|\sup_{s, t\in[0,1]} \frac{| f(t) - f(s)|}{|t-s|^\alpha} 
\leq \frac{\delta }{2} \|f\|_\alpha.
\end{align*}
\end{example}
These conditions are slightly stronger than those of \cite{baldi1986large} and are used in \cite{eddahbi2002strassen}. Condition 2. in Definition \ref{dfn:SystemContraction} needs to be strengthened to allow it to be applied to H\"older norms rather than just supremum norms. Observe that by choosing $y_2 = z_2 = x$, one gets 
$$
| \Gamma_\alpha(y_1)  - \Gamma_\alpha(z_1) | \leq |\Gamma_\beta(y_1) - \Gamma_\beta(z_1) |.
$$ 
This stronger condition still allows for the example of linear contractions up to a transformation.  For $s\in \bR^+$ define 
$$
\phi(s) = \sqrt{s \log(\log(s))}.
$$
Let $b:\bR^d \times \cP_2(\bR^d) \to \bR^d$ and $\sigma:\bR^d\times \cP_2(\bR^d) \to \bR^{d \times m}$ be progressively measurable functions such that there is a unique solution to
$$
dY(t) = b(Y(t), \cL_t^Y) dt + \sigma(Y(t), \cL_t^Y) dW(t), \qquad Y(0) = x\in \bR^d. 
$$


\begin{definition}
Let $u>3$. Let $\hat{\sigma}_u : \bR^d \times \cP_2(\bR^d) \to \bR^{d\times d'}$ and $\hat{b}_u : \bR^d \times \cP_2(\bR^d) \to \bR^{d}$ be such that
\begin{align*}
\hat{\sigma}_u(y, \mu) &= \phi(u) \nabla \Big[\Gamma_{\phi(u)}\Big]\Big( \Gamma_{\phi(u)^{-1}}(y) \Big)^T \sigma\Big(\Gamma_{\phi(u)^{-1}}(y), \mu\circ \Gamma_{\phi(u)} \Big)
\\
\hat{b}_u(y, \mu) &= u \textbf{L}(y, \mu)\Big[ \Gamma_{\phi(u)} \Big] \Big( \Gamma_{\phi(u)^{-1}}(y) \Big),
\end{align*}
where for $\tilde{a} = \sigma^T \sigma$ the operator $\textbf{L}(\cdot, \cdot)[\cdot]$ is given as 
\begin{align*}
\textbf{L}(y, \mu) \Big[ f\Big]\Big(z\Big) =& \sum_{i=1}^d \frac{\partial f}{\partial y_i} \Big(\Gamma_{\phi(u)^{-1}}(z)\Big) b_i\Big( \Gamma_{\phi(u)^{-1}}(y), \mu\circ\Gamma_{\phi(u)} \Big) 
\\
&+ \frac{1}{2}\sum_{i, j=1}^d \tilde{a}_{i, j}\Big( \Gamma_{\phi(u)^{-1}}(y), \mu\circ\Gamma_{\phi(u)} \Big) \frac{\partial^2 f
}{\partial y_i \partial y_j}\Big(\Gamma_{\phi(u)^{-1}}(z)\Big),
\end{align*}

\end{definition}

\begin{assumption}
\label{GammaSystemContractions}
Throughout we assume that $\Gamma_u$ is twice differentiable for all $u>3$ and that $\forall y \in \bR^d$, $\forall \mu\in \cP_2(\bR^d)$ we have for some $\hat{\sigma} : \bR^d \times \cP_2(\bR^d) \to \bR^{d\times d'}$ and $\hat{b} : \bR^d \times \cP_2(\bR^d) \to \bR^{d}$
\begin{align*}
\lim_{u\to \infty} \hat{\sigma}_u(y, \mu) = \hat{\sigma}(y, \mu)
\quad\textrm{and}\quad 
\lim_{u\to \infty} \hat{b}_u(y, \mu) = \hat{b}(y, \mu),
\end{align*}
where $\hat{\sigma}$ and $\hat{b}$ satisfy Assumption \ref{ass:MKSDE-MainExistTheo} with the addition that $\hat{\sigma}$ is bounded by constant $M$. 
\end{assumption}

For $t, u \in \bR^+$ define 
$$
Z_u(t) = \Gamma_{\phi(u)} \big(Y(ut)\big), 
$$
and recall that since $Y(0)=x$ and $\Gamma_u(x) = x$, by assumption $Z_u(0) = x$. We use It\^o's formula on $Z_u(t)$ by assuming twice differentiability of $\Gamma_{\phi(u)}(\cdot)$. 
\begin{align*}
dZ_u(t) = d( \Gamma_{\phi(u)} ( Y(ut) )
= \nabla \Big[\Gamma_{\phi(u)}\Big]& \Big( Y(ut) \Big)^T dY(ut) 
\\
&+ \frac{dY(ut)^T}{2} H\Big[\Gamma_{\phi(u)}\Big] \Big( Y(ut) \Big) dY(ut). 
\end{align*}
Rewriting $Y(ut) = \Gamma_{\phi(u)^{-1}}(Z_u(t))$ and substituting in gives
\begin{align*}
dZ_u(t) =& u \sum_{i=1}^d \frac{\partial \Gamma_{\Phi(u)}}{\partial y_i}\Big( \Gamma_{\phi(u)^{-1}}(Z_u(t)) \Big) b_i\Big( \Gamma_{\phi(u)^{-1}}(Z_u(t)), \cL_t^{Z_u}) \circ \Gamma_{\phi(u)} \Big)  dt
\\
&+ \frac{u}{2} \sum_{i, j=1}^d \frac{\partial^2 \Gamma_{\Phi(u)}}{\partial y_i \partial y_j}\Big( \Gamma_{\phi(u)^{-1}}(Z_u(t)) \Big) \sum_{k=1}^{d'} \sigma_{k, i} \sigma_{j, k} \Big( \Gamma_{\phi(u)^{-1}}(Z_u(t)), \cL_t^{Z_u}) \circ \Gamma_{\phi(u)} \Big)  dt
\\
&+ \sum_{i=1}^d \frac{\partial \Gamma_{\Phi(u)}}{\partial y_i}\Big( \Gamma_{\phi(u)^{-1}}(Z_u(t)) \Big) \sum_{k=1}^{d'} \sigma_{i, k}\Big( \Gamma_{\phi(u)^{-1}}(Z_u(t)), \cL_t^{Z_u}) \circ \Gamma_{\phi(u)} \Big)dW_k(ut). 
\end{align*}
Next, using that $\cW_u (t) = \frac{W(ut)}{\sqrt{u}}$ is a Brownian motion, we can rewrite all of this as the SDE with initial condition $Z_u(0) = x$
$$
dZ_u(t) = \frac{1}{\sqrt{\log\log(u)}} \hat{\sigma}_u \Big(Z_u(t), \cL_t^{Z_u})\Big) d\cW_u(t) + \hat{b}_u \Big(Z_u(t), \cL_t^{Z_u})\Big) dt.
$$
Under Assumption \ref{GammaSystemContractions} and using Theorem \ref{theorem:LDPResult} we get
\begin{align}
\label{eq:LDP}
-\Delta(\mathring{A}) &\leq \liminf_{u\to \infty}\frac{1}{\log\log(u)} \log \bP(Z_u\in A) \nonumber
\\
& \leq \limsup_{u\to \infty} \frac{1}{\log\log(u)} \log \bP(Z_u\in A) \leq -\Delta(\bar{A})
\end{align}
for every Borel set $A$ induced by the $\alpha$-H\"older topology with $\alpha <1/2$. Recall the definition of the rate function $\Delta(A):=\inf\big\{ {\|\dot h\|_2^2}/{4} ; h\in H^{\otimes d'}, \Phi^x(h)(\cdot)\in A\big\}$ with Skeleton Process
$$
\Phi^x(h)(t)= x + \int_0^t \hat{b}(\Phi^x(h)(s), \delta_{\Phi^x(0)(s)}) ds + \int_0^t \hat{\sigma}(\Phi^x(h)(s), \delta_{\Phi^x(0)(s)}) \dot{h}(s) ds. 
$$

We can now state the main result of this section.
\begin{theorem}
\label{theorem:StrassensLaw}
With probability 1, the set of paths $\{Z_u; u>3\}$ is relatively compact on the H\"older topology and its set of limit points coincides with $K=\{\Phi(h): \frac{||\dot{h}||_2^2}{2}\leq 1 \}$. 
\end{theorem}
We first prove some technical Lemmas. 
\begin{lemma}
\label{lemma:5.1}
$\forall c>1$ and $\forall \varepsilon>0$ there exists a positive integer $j_0(\omega)$ almost surely finite such that $\forall j>j_0$
\begin{align*}
d_\alpha (Z_{c^j} , K)< \sqrt{\varepsilon},
\quad \textrm{where}\quad 
d_\alpha (x,A) = \inf \Big\{ \|x - y \|_\alpha : y\in A \Big\}.
\end{align*}
\end{lemma}

\begin{proof}
Start by considering the set of $\alpha$-H\"older continuous paths $C_\varepsilon:= \{g; d_\alpha(g, K)\geq \sqrt{\varepsilon} \}$. By definition we have that $\Delta(C_\varepsilon)>1$, so there exists a real number $\delta>0$ such that $\Delta(C_\varepsilon)>1+\delta$. Using the LDP results in \eqref{eq:LDP}, we can rearrange this to get
$$
\bP\big[ Z_{c^j} \in C_\varepsilon \big] 
\leq 
\exp\Big( -(1+\delta) \log\log(c^j) \Big) \lesssim \frac{1}{j^{1+\delta}} . 
$$
Clearly $\sum_{j=1}^\infty \bP\big[ Z_{c^j} \in C_\varepsilon \big] <\infty$ and by a direct application of Borel-Cantelli we have 	
$
\bP\big[ d_\alpha (Z_{c^j}, K)> \sqrt{\varepsilon} \mbox{ i.o.} \big] = 0
$.
\end{proof}

\begin{lemma}\label{lemma:5.2}
$\forall \varepsilon>0$ $\exists c_\varepsilon>1$ such that for $1<c<c_\varepsilon$ there exists an almost surely finite integer $j_0(\omega)$ such that $\forall j>j_0$, $A_{j, c}\leq \sqrt{\varepsilon}$. 
\end{lemma}

\begin{proof}
For notational convenience define, for $c>1$ and for every positive integer $j$, the quantity
$$
A_{j, c} = \sup_{c^{j-1} \leq u \leq c^j} \|Z_u - \Gamma_{\phi(u)} \circ \Gamma_{\phi(c^j)^{-1}} (Z_{c^j}) \|_\alpha.
$$
Start by observing that the set $K$ is relatively compact in the $\alpha$-topology, so it is bounded. Therefore, by Lemma \ref{lemma:5.1}, we have that $\forall j>j_0$ that $||Z_{c^j}||_\alpha <C$. We want to show that
$$
\sum_{j\geq1} \bP\Big[ A_{c, j} >\sqrt{\varepsilon} \Big] < \infty
\quad \textrm{which is equivalent to}\quad 
\sum_{j>j_0} \bP\Big[ A_{c, j} >\sqrt{\varepsilon}, ||Z_{c^j}||_\alpha<C \Big] < \infty. 
$$
Considering one of these sets, we see
\begin{align*}
&\{ A_{j, c}>\sqrt{\varepsilon} , ||Z_{c^j}||_\infty \leq C\} 
\\
&= \Big\{ \sup_{c^{j-1} \leq u \leq c^j} \sup_{0\leq s, t\leq 1}\frac{| \Gamma_{\phi(u)}(Y(ut)) - \Gamma_{\phi(u)}(Y(c^jt)) - \Gamma_{\phi(u)}(Y( us) + \Gamma_{\phi(u)}(Y(c^js)) |}{|t-s|^\alpha } >\sqrt{\varepsilon}, 
\\
&\hspace{30pt} ||Z_{c^j}||_\alpha < C\Big\}.
\end{align*}
Using Definition \ref{dfn:SystemContraction}, for $u\in[c^{j-1}, c^{j+1}]$
\begin{align*}
|\Gamma_{\phi(u)}&(Y(ut)) - \Gamma_{\phi(u)}(Y(c^jt)) - \Gamma_{\phi(u)}(Y( us) + \Gamma_{\phi(u)}(Y(c^js)) |
\\
&\leq |\Gamma_{\phi(c^{j-1})}(Y(ut)) - \Gamma_{\phi(c^{j-1})}(Y(c^jt)) - \Gamma_{\phi(c^{j-1})}(Y( us) + \Gamma_{\phi(c^{j-1})}(Y(c^js)) |.
\end{align*}
Therefore

\begin{align*}
&\{ A_{c, j} >\sqrt{\varepsilon} \} 
\\
&\subseteq\Bigg\{ \sup_{\frac{1}{c}\leq v\leq 1} \sup_{0\leq s, t\leq 1} \frac{1}{|t-s|^\alpha} \Big|\Gamma_{\phi(c^{j-1})}(Y(c^j vt)) - \Gamma_{\phi(c^{j-1})}(Y(c^j t)) 
\\
&\hspace{100pt} - \Gamma_{\phi(c^{j-1})}(Y( c^j vs) + \Gamma_{\phi(c^{j-1})}(Y(c^j s)) \Big|
 >\sqrt{\varepsilon} \Bigg\}
\\
&\subseteq \Bigg\{ \sup_{\frac{1}{c}\leq v\leq 1} \sup_{0\leq s, t\leq 1} \frac{|Z_{c^j}(vt) - Z_{c^j}(t) - Z_{c^j} (vs) + Z_{c^j}(s) |}{|t-s|^\alpha} >\frac{{\sqrt{\varepsilon}}}{2} \Bigg\},
\end{align*}
using that $\exists j$ large enough so that for and $\delta>0$
$$
\frac{\phi(c^{j-1})}{\phi(c^{j})} = \frac{1}{\sqrt{c}} \sqrt{\frac{\log\log(c^j)}{\log\log(c^{j-1})}} \leq \frac{1}{\sqrt{c}} (1-\delta),
$$
and choosing $c$ small enough we can make $\Gamma_{\frac{\phi(c^{j-1})}{\phi(c^j)}}$ within $\tfrac{{\sqrt{\varepsilon}}}{2}$ of the identity operator using properties from Definition \ref{dfn:SystemContraction}.  Therefore
\begin{align*}
&\{ A_{c, j} >{\sqrt{\varepsilon}}, ||Z_{c^j}||_\infty \leq C \} 
\\
&\subseteq \Bigg\{ \sup_{\frac{1}{c}\leq v\leq 1} \sup_{0\leq s, t\leq 1} \frac{|Z_{c^j}(vt) - Z_{c^j}(t) - Z_{c^j}(vs) + Z_{c^j}(t)|}{|t-s|^\alpha} >\frac{{\sqrt{\varepsilon}}}{2}, ||Z_{c^j}||_\alpha \leq C \Bigg\}
\\
&\subseteq \{ Z_{c^j} \in b_{\varepsilon}\},
\end{align*}
where the set $b_{\varepsilon}$ is given by
$$
b_{\varepsilon} = \Bigg\{g\in C^\alpha([0,1]; \bR^d): \sup_{\frac{1}{c}\leq v\leq 1} \sup_{0\leq s, t\leq 1} \frac{|g(vt) - g(t) - g(vs) + g(t)|}{|t-s|^\alpha}  >\frac{{\sqrt{\varepsilon}}}{2}, ||g||_\alpha \leq C \Bigg\},
$$
as we would expect.  Let $h\in H^{\otimes d'}$ so $\|\dot{h}\|_2<\infty$ such that $\Phi(h)\in b_{\varepsilon}$, then
\begin{align}
\frac{{\sqrt{\varepsilon}}}{2} |t-s|^\alpha &\leq \Big| [ \Phi(h)(t) - \Phi(h)(vt) ] - [\Phi(h)(s) - \Phi(h)(sv)]\Big| \nonumber
\\
\label{eq:Lemma2.1}
&\leq \Big| \int_{(vt) \vee s}^t d\Phi(h)(r) - \int_{vs}^{s\wedge (tv)} d\Phi(h)(r) \Big|,
\end{align}
for at least some choice of $v\in [\tfrac{1}{c}, 1]$ and $t, s\in [0,1]$. 

We know that a solution to the ODE $\Phi(h)$ exists uniquely  and has finite supremum. Therefore we can easily conclude that there exists constants $M_1$ and $M_2$ such that
\begin{align*}
\Big| \int_s^t d\Phi(h)(r) \Big| &\leq \Big| \int_s^t b(\Phi(h)(r), \delta_{\Phi(h)(r)}) dr + \int_s^t \sigma( \Phi(h)(r), \delta_{\Phi(h)(r)}) dh(r) \Big|
\\
&\leq M_1 \sqrt{|t-s|} ||\dot{h}||_2 + M_2 |t-s|.
\end{align*}
It follows from \eqref{eq:Lemma2.1} that
$$
\|\dot{h}\|_2 
\geq \frac{\frac{{\sqrt{\varepsilon}}}{2} |t-s|^\alpha - M_2 \Big( \Big| t-s\vee(tv) \Big| + \Big| s\wedge(tv)-(sv) \Big|\Big)}{M_1 \Big( \Big| t-s\vee(tv) \Big|^{\frac{1}{2}} + \Big| s\wedge (tv)-(sv) \Big|^{\frac{1}{2}} \Big)}. 
$$
Let us consider first the case where $s<(tv)$. 
\begin{align*}
\|\dot{h}\|_2 
&
\geq \frac{\frac{{\sqrt{\varepsilon}}}{2} \Big|t-s \Big|^\alpha - M_2 \Big|(t+s)(1-v) \Big|}{M_1 \Big| (\sqrt{t} + \sqrt{s}) \sqrt{1-v} \Big|}
\\
&
\geq \frac{{\sqrt{\varepsilon}} |1-\tfrac{1}{c}|^\alpha}{4M_1 | \sqrt{1-\tfrac{1}{c}} |} - \frac{M_2}{M_1}\sqrt{1-\tfrac{1}{c}}, 
\end{align*}
so for $c$ small enough we have 
$
\|\dot{h}\|_2 \geq 1+ \delta
$ 
for any choice of $\delta>0$. 

Secondly, consider the case where $s>(tv)$

\begin{align*}
\|\dot{h}\|_2 
&
\geq \frac{\frac{{\sqrt{\varepsilon}}}{2} \Big|t(1-\frac{s}{t})\Big|^\alpha - M_2 \Big( \Big| t(1-\frac{s}{t}) \Big| (1+v) \Big)}{2M_1 \Big( \Big| t(1-\frac{s}{t})  \Big|^{\frac{1}{2}}  \Big)}
\\
&
\geq \frac{{\sqrt{\varepsilon}} |1-\tfrac{1}{c}|^\alpha}{4M_1 | \sqrt{1-\tfrac{1}{c}} |} - \frac{M_2}{M_1}\sqrt{1-\tfrac{1}{c}}, 
\end{align*}
and taking $c>1$ small enough as before gives 
$
\|\dot{h}\|_2 \geq 1+ 2\delta. 
$ 

Therefore, using Equation \eqref{eq:LDP} we can get
\begin{align*}
\bP[ Z_{c^j} \in b_{\varepsilon} ] 
&\leq \exp\Big( -(\Delta(b_{\varepsilon}) - \delta) \log\log(c^j)\Big)
\\
&
\leq \exp\Big( -(1 + \delta) \log\log(c^j)\Big)
\lesssim \frac{1}{j^{1+\delta}}, 
\end{align*}
and the conclusion of the proof is straightforward by Borel Cantelli. 
\end{proof}

We are now able to prove the main theorem. 

\begin{proof}[Proof of Theorem \ref{theorem:StrassensLaw}]
The proof is divided into two parts: 

\emph{Step 1. Relative Compactness.} For any $c>1$, there will exist $j\in \bN$ such that $c^{j-1}< u < c^{j}$
\begin{align}
\label{eq:5.Thm1}
&d_\alpha(Z_u, K) \leq d_\alpha (Z_{c^j}, K)
\\
\label{eq:5.Thm2_Thm3}
&\quad + || \Gamma_{\phi(u)} \circ \Gamma_{\phi(c^j)^{-1}} (Z_{c^j}) - Z_{c^j})||_\alpha
+|| Z_u - \Gamma_{\phi(u)} \circ \Gamma_{\phi(c^j)^{-1}} (Z_{c^j}) ||_\alpha,
\end{align}
where $j$ is chosen so that $c^{j-1} \leq u \leq c^j$. 

Lemma \ref{lemma:5.1} with $j$ large enough ensures that \eqref{eq:5.Thm1} is bounded by $\frac{{\sqrt{\varepsilon}}}{3}$. From Lemma \ref{lemma:5.1}, we have that $Z_{c^j}$ is bounded, since $\forall \delta>0$,
\begin{align*}
1 \geq \frac{\phi(u)}{\phi(c^{j})} \geq \frac{\phi(c^{j-1})}{\phi(c^j)} \geq  \frac{(1-\delta)}{\sqrt{c}},
\end{align*}
for $j$ large enough. Choosing $1<c$ small enough, we can use the forth part of Definition \ref{dfn:SystemContraction} to get that the 1st term in \eqref{eq:5.Thm2_Thm3} is less than $\frac{{\sqrt{\varepsilon}}}{3}$. Lemma \ref{lemma:5.2} bounds the 2nd term of \eqref{eq:5.Thm2_Thm3} by $\frac{{\sqrt{\varepsilon}}}{3}$. 

Therefore, we conclude that the set $\{ Z_u: u>3\}$ is relatively compact (and hence we have convergence in probability). 

\emph{Step 2. The set of limit points.} Let $\Phi(h)\in K$ so that $\frac{||\dot{h}||_2^2}{2}<1$. Then for $\varepsilon>0$ and $\beta>0$, we define the sets
$$
E_j 
=\Big\{ \Big\| \frac{\cW_{c^j} (t)}{\sqrt{ \log\log(c^j)}} - h \Big\|_\infty \leq \beta \Big\} 
\quad\textrm{and}\quad
F_j 
= \Big\{ \Big\| Z_{c^j} - \Phi(h)\Big\|_\alpha \leq {\sqrt{\varepsilon}} \Big\}.
$$
Using Proposition \ref{HolderProp}, we have that for $j$ large enough and $\alpha$ small enough that
\begin{equation}
\label{eq:5.Lemma3}
\bP[E_j] - \bP[F_j] = \bP \Big[ E_j \cap F_j^c \Big] \leq \exp\Big( -2 \log \log(c^j)\Big) \lesssim \frac{1}{j^2}.
\end{equation}
Strassen's Law tells us that $\bP(\limsup E_j) = 1$, see \cite{strassen1964invariance}. Therefore
$
\sum_{j} \bP [E_j ] = \infty.
$
However, by Equation \eqref{eq:5.Lemma3} we also have
\begin{align*}
\sum_{j} \Big(\bP[E_j] - \bP[F_j]\Big) < \infty
&\quad \Rightarrow \quad 
\sum_{j} \bP\Big[F_j\Big] = \infty
\\
&
\quad \Rightarrow \quad 
\bP\Big[ \Big\| Z_{c^j}  - \Phi(h) \Big\|_\alpha < {\sqrt{\varepsilon}} \mbox{ i.o.}\Big] = 1,
\end{align*}
the latter following from Borel-Cantelli.

Finally since $(c^j_{j\in\bN})$ is just a subsequence of $(m)_{m\in\bN}$, the result can be extended to the conclusion. 

\end{proof}

\appendix

\section{A collection of auxiliary results}


\subsection{Classical Large Deviation Principles}
The following lemma corresponds to \cite[Lemma 5.6.18]{DemboZeitouni2010}.
\begin{lemma}
\label{eurodance}
Let $(b_t)_t,(\sigma_t)_t$ be progressively measurable processes. Let
\begin{equation}
\label{truc}
dz_t=b_tdt+{\sqrt{\varepsilon}}\sigma_tdW(t)\,\quad \textrm{where $z_0$ is deterministic.}
\end{equation}
Let $\tau\in[0,1]$ be a stopping time with respect to the filtration of $\left(W(t)\right)_{t\in[0,1]}$. Suppose that the coefficients of the diffusion matrix are uniformly bounded, and for some constants $M$, $B$, $\rho$ and any $t\in[0, \tau_1]$,
\begin{equation*}
\left|\sigma_t\right|\leq M\left(|z_t|^2+\rho^2\right)^{\frac{1}{2}}\,,
\qquad\textrm{and}\qquad
\left|b_t\right|\leq B\left(|z_t|^2+\rho^2\right)^{\frac{1}{2}}\,.
\end{equation*}
Then, for any $\delta>0$ and any $\varepsilon\leq1$,

\begin{equation*}
\varepsilon\log\Big(\mathbb{P}\big[\sup_{[0,\tau_1]}|z_t|\geq\delta\big]\Big)\leq B+M^2\left(1+\frac{d}{2}\right)+\log\left(\frac{\rho^2+|z_0|^2}{\rho^2+\delta^2}\right)\,.
\end{equation*}
In particular, if $z_0=0$:
\begin{equation*}
\varepsilon\log\Big(\mathbb{P}\big[\sup_{[0,\tau_1]}|z_t|\geq\delta\big]\Big)\leq B+M^2\left(1+\frac{d}{2}\right)+\log\left(\frac{\rho^2}{\rho^2+\delta^2}\right)\,.
\end{equation*}

\end{lemma}

We also need \cite[Lemma 5.2.1]{DemboZeitouni2010}.
\begin{lemma}
\label{camion}
For any dimension $d'$, and any $\tau,\varepsilon,\delta$,
\begin{equation*}
\mathbb{P}\left[\sup_{0\leq t\leq \tau}\left|{\sqrt{\varepsilon}} W(t)\right|\geq\delta\right]\leq4d'\exp\left(-\frac{\delta^2}{2d\tau\varepsilon}\right)\,.
\end{equation*}
\end{lemma}

We also need \cite[Theorem 4.2.23]{DemboZeitouni2010}.
\begin{proposition}
\label{blablabla}
Let $\left\{\mu_\varepsilon\right\}$ be a family of probability measures that satisfies the LDP with a good rate function $I$ on a Hausdorff topological space $\mathcal{X}$, and for $m\in\bN$, let $f_m$ be continuous functions from $\mathcal{X}$ to $\mathcal{Y}$, where $(\mathcal{Y},d)$ is a metric space. Assume there exists a measurable map $f$ from $\mathcal{X}$ to $\mathcal{Y}$ such that for every $\alpha<\infty$ :
\begin{equation*}
\limsup_{m\to\infty}\sup_{x\,:\,I(x)\leq\alpha}\,d\big(f_m(x),f(x)\big)=0\,.
\end{equation*}
Then any family of probability measures $\left\{\widetilde{\mu_\varepsilon}\right\}$ for which $\left\{\mu_\varepsilon\circ f_m^{-1}\right\}$ are exponentially
good approximations satisfies the LDP in $\mathcal{Y}$ with the good rate function $I'(y):=\inf\{I(x)\,:\,y = f(x)\}$.
\end{proposition}

\subsection{Large Deviation Principles in path space topologies}
The following results are of their own independent interest and can be found in \cite[\emph{Lemme} 1 p.196]{arous1994grandes} with token proofs. We provide a full proof for the benefit of the reader. The extension to $[0,T]$ is straightforward.
\begin{lemma}[\cite{arous1994grandes}]
\label{Lemme1}
Let $(W(t))_{t\in[0,1]}$ be a $d'$-dimensional Brownian motion. Then, there exists a constant $C>0$ which is independent of $m$ such that $\forall u, v>0$
$$
\bP\Big[ \|W\|_\alpha \geq u, \|W\|_\infty\leq v\Big] 
\leq 
C \max\Big(1, \Big(\frac{u}{v}\Big)^{1/ \alpha}\Big) \exp\Big( \frac{-1}{C} \frac{u^{1/\alpha}}{v^{(1/\alpha)-2}}\Big).
$$
\end{lemma}
\begin{proof}[Proof of Lemma \ref{Lemme1}]
Consider a Brownian motion $W$ satisfying the constraint $\|W\|_\infty\leq v$. We use methods from \cite{herrmann2013stochastic} to represent the $\alpha$-H\"older norm in terms of a supremum of Fourier coefficients generated by Schauder functions. By direct calculation, one can dominate the Fourier coefficients 
$$
|W_{pm}|=2^{p/2}\Big|2W\Big(\frac{2m-1}{2^{p+1}}\Big)-W\Big(\frac{m}{2^p}\Big)-W\Big(\frac{m-1}{2^p}\Big)\Big| \leq 2^{p/2} 4v. 
$$
If we also restrict that $\|W\|_\alpha = \sup_{p,m} |W_{pm}| 2^{p(\alpha-1/2)}\geq u$ and search for values of $p$ and $m$ which do not yield a contradiction. Observe that we require $u\geq 4v 2^{\alpha p}$. If we consider a $p$ where this was not true, we would have that $W_{pm}<u$. The supremum of all $W_{pm}$ is still be greater than $u$, but this value of $p$ could be removed from the collection over which the supremum is taken over without affecting the measure of the event. Let $p_0$ be the least such relevant $p$, defined as $p_0:=\inf \left\{ p\in \bN; 2^{\alpha p}\geq {u}/({4v})\right\}$. Then for an arbitrary choice of $\lambda>0$, we have
\begin{align*}
&\bP[ \|W\|_\alpha \geq u, \|W\|_\infty\leq v] 
\\
&\qquad 
=\bP[ \sup_{p\geq p_0, m} 2^{p(\alpha-1/2)} |W_{pm}| \geq u] 
\leq \frac{ \sup_{p\geq p_0} \bE[ \exp(\lambda 2^{p(\alpha-1/2)} |W_{pm}|)]}{\exp(\lambda u)} 
\\
&\qquad
\leq \sup_{p\geq p_0} 2\exp\Big( \frac{\lambda^2 2^{p(2\alpha-1)}}{2} -\lambda u\Big) 
\leq 2 \exp\Big( \frac{-u^2 2^{p_0(1-2\alpha)}}{2}\Big),
\end{align*}
where for the last line we choose $\lambda=u2^{p(1-2\alpha)}$ to minimize the expression (since $\lambda$ is arbitrary). From the definition of $p_0$ we have
$$
2^{p_0(1-2\alpha)} \geq \Big(\frac{u}{4v}\Big)^{\frac{1}{\alpha}-2},
$$
and substituting this in yields the final result. 
\end{proof}
The next lemma iterates on the first, see \cite[\emph{Lemme 2}, p.196]{arous1994grandes}.
\begin{lemma}
\label{Lemme2}
Let $(W(t))_{t\in [0,1]}$ be a $d'$-dimensional Brownian motion. There exists a constant $C'>0$ which is independent of $d'$ and $\alpha$ such that $\forall u>0$ and $\forall K\in C([0,1])$ such that $\|K\|_\infty\leq 1$
$$
\bP\Big[ \|\int_0^\cdot K(s) dW(s)\|_\alpha \geq u, \|K\|_\infty\leq 1\Big] \leq C' \exp\Big( \frac{-u^2}{C'}\Big).
$$
\end{lemma}
\begin{proof}[Proof of Lemma \ref{Lemme2}]
Let $\|K\|_\infty\leq 1$. In the case where $K$ is deterministic, the stochastic integral of $K$ is clearly normally distributed and the result is clear. For $K$ not deterministic it hard to say anything about the probability distribution of the stochastic integral. 

Using the equivalent definition of H\"older norms in terms of a Schauder expansion, we have that 
\begin{align*}
&\bP \Big[ \|\int_0^\cdot K(s) dW(s)\|_\alpha \geq u, \|K\|_\infty \leq 1\Big]
\\
&\qquad 
= \bP \Big[ \sup_{p, m} \Big|\int_0^1 H_{pm}(s)K(s)dW(s)\Big| \geq u, \|K\|_\infty \leq 1\Big]
\\
&\qquad 
\leq \frac{\bE \Big[  \exp\Big( \lambda  \sup_{p, m}\Big|\int_0^1 H_{pm}(s)K(s)dW(s)\Big|\Big) \Big]}{\exp(\lambda u)} 
\\
&\qquad 
\leq  \sup_{p, m} \frac{\bE \Big[  \exp\Big( \lambda \Big|\int_0^1 H_{pm}(s)K(s)dW(s)\Big|\Big) \Big]}{\exp(\lambda u)} 
\end{align*}
where the supremum can come outside the expectation by Beppo Levi Theorem since the random variables are all positive. Temporarily, consider the process $Y(t)=\int_0^t H_{pm}(s)K(s)dW(s)$.  
Using It\^o's formula we get that
$$
\bE[|Y(t)|^n] = \int_0^t \bE\Big[ \frac{n(n-1)}{2} |Y(s)|^{n-2} H_{pm}(s)^2K(s)^2 \Big] ds. 
$$
$Y(t)$ is a martingale, since $H_{pm}$ and $K$ are bounded, with $Y(0)=0$ so $\bE[Y(t)]=0$. By the It\^o Isometry, the second moment of $Y(t)$ is
\begin{align*}
\bE[Y(t)^2] &= \bE\Big[\int_0^t H_{pm}(s)^2K(s)^2ds\Big]
\leq 	\begin{cases}
		0 & 0\leq t\leq\frac{m-1}{2^p}\\
		2^p(t-\frac{m-1}{2^p})& \frac{m-1}{2^p} < t< \frac{m}{2^p}\\
		1 & \frac{m}{2^p} \leq t\leq 1
		\end{cases}
\end{align*}
Therefore, by induction on $n$ we see
\begin{align*}
\bE[Y(t)^{2n}] \leq \begin{cases}
		0 & 0\leq t\leq\frac{m-1}{2^p}\\
		\frac{(2n)!}{n!2^n} (2^p)^n (t-\frac{m-1}{2^p})^n & \frac{m-1}{2^p} < t< \frac{m}{2^p}\\
		\frac{(2n)!}{n!2^n} & \frac{m}{2^p} \leq t\leq 1
				\end{cases}.
\end{align*}
For the odd moments of $|Y(t)|$, we first use the Burkholder-Davies-Gundy Inequality to say
\begin{align*}
\bE[|Y(t)|] 
& 
\leq C_1\bE\Big[ \Big(\int_0^t H_{pm}(s)^2K(s)^2 ds\Big)^{\frac{1}{2}} \Big] 
\\
&
\leq \begin{cases}
		0 & 0\leq t\leq\frac{m-1}{2^p}\\
		C_1 2^{p/2} (t-\frac{m-1}{2^p})^{1/2}& \frac{m-1}{2^p} < t< \frac{m}{2^p} \\
		C_1 & \frac{m}{2^p} \leq t\leq 1
		\end{cases},
\end{align*}
and by induction on $n$ again we see that
\begin{align*}
\bE[|Y(t)|^{2n+1}] \leq \begin{cases}
		0 & 0\leq t\leq\frac{m-1}{2^p}\\
		C_1 n!2^n (t-\frac{m-1}{2^p})^{\frac{2n+1}{2}} 2^{\frac{2n+1}{2}} & \frac{m-1}{2^p} < t< \frac{m}{2^p} \\
		C_1 n! 2^n & \frac{m}{2^p} \leq t\leq 1
		\end{cases}
\end{align*}
Hence $\bE[|Y(t)|^{2n}]\leq (C_1\vee 1) \frac{(2n)!}{n! 2^n}$ and $\bE[|Y(t)|^{2n+1}]\leq (C_1\vee 1) n! 2^n$. 
The upper bounds for these moments are the same as the moments of a Half normal distribution with variance $1$ up to a multiplicative constant. Therefore, we can upper bound the moment generating function of the RV $|Y(1)|$ using the moment generating function of a half normal random variable. If $Z$ is half normally distributed with variance $a$, we have
\begin{align*}
\bE\Big[\exp(\lambda Z)\Big] &= \int_0^\infty \frac{2}{\sqrt{2\pi a^2}} e^{\lambda x} \exp\Big( \frac{-x^2}{2a^2}\Big) dx 
\leq 4 \exp\Big( \frac{\lambda a^2}{2}\Big).
\end{align*}
Therefore 
 $\bE\Big[ \exp\Big( \lambda \Big| \int_0^1 H_{pm}(s)K(s)dW(s)\Big| \Big) \Big] \lesssim \exp\Big(\frac{\lambda^2}{2}\Big)$
and hence
\begin{align*}
\bP\Big[ \|\int_0^\cdot K(s) dW(s)\|_\alpha \geq u, \|K\|_\infty\leq 1\Big] &\lesssim \exp\Big(\frac{\lambda^2}{2}-\lambda u\Big) 
\lesssim \exp\Big( \frac{-u^2}{2}\Big),
\end{align*}
by choosing $\lambda$ to minimize the equation since the choice of $\lambda$ was arbitrary ($\lambda=u$). 
\end{proof}
\begin{lemma}
\label{Lemma3}
Let $\psi\in C^\alpha([0,1])$ 
 with $\psi(0)=0$. Then $\|\psi\|_\infty \leq \|\psi\|_{\alpha}$. 
\end{lemma}
\begin{proof}
Using that $t\in[0,1]$ one easily computes
\begin{align*}
\|\psi\|_{\infty} 
= 
\sup_{t\in[0,1]} |\psi(t)| 
&
=
\sup_{t\in[0,1]} \frac{|\psi(t)-\psi(0)|\cdot |t-0|}{|t-0|} 
\\
&
\leq
 \sup_{t\in[0,1]} \frac{|\psi(t)-\psi(0)|}{|t-0|} \cdot \sup_{t\in[0,1]} |t-0| 
\leq \|\psi\|_\alpha.
\end{align*}
\end{proof}


\section*{Acknowledgements}
The authors thank the referee for an in-depth review of the initial version of our manuscript which lead to several improvements.



\end{document}